\documentclass[12pt,makeidx]{amsart}
\usepackage{amsmath,amsthm,amsopn,amssymb,a4wide,array,delarray,times}

\usepackage{mathrsfs}
\usepackage{amssymb}
\usepackage{txfonts}

\usepackage{enumerate,color}
\usepackage[pagebackref,colorlinks,linkcolor=red,citecolor=blue,urlcolor=blue,hypertexnames=true]{hyperref}
\usepackage{url}

\renewcommand{\subset}{\subseteq}
\newtheorem{theorem}            {Theorem}[section]
\newtheorem{corollary}          [theorem]{Corollary}
\newtheorem{proposition}        [theorem]{Proposition}

\newtheorem{lemma}              [theorem]{Lemma}

\newtheoremstyle{note}
     {}
     {}
     {}
     {}
     {\bfseries} 
     {.}
     {.5em}
     {}

\theoremstyle{note}

\newtheorem{remark}[theorem]{Remark}

\newtheorem{example}[theorem]{Example}

\newcommand{\df}[1]{{\it{#1}}{\index{#1}}}

\def\rR{\mathfrak R}

\def\cL{\mathcal L}
\def\cX{\mathcal X}
\def\cC{\mathcal C}
\def\cA{\mathscr{A}}

\def\AD{\mathbb A(\mathbb D)}

\def\fs{\mathcal F}
\def\fsg{\mathscr F}
\def\fn{N}

\def\bfs{\mathfrak{F}}

\def\fz{\mathfrak{z}}
\def\fw{\mathfrak{w}}

\def\At{\tilde{A}}

\def\Dinf{\mathbb D_\infty}

\def\oD{\overline{\mathbb D}}
\newcommand{\ip}[2]{{\left<#1,#2\right>}}

\begin{document}

\title[Dilations and Constrained Algebras]{Dilations and Constrained Algebras}

\author[Dritschel]{Michael A.~Dritschel}
\address{School of Mathematics {\&} Statistics\\
  Newcastle University \\
  Newcastle upon Tyne\\
  NE1 7RU\\
  UK
   }
   \email{michael.dritschel@ncl.ac.uk}

\author[Jury]{Michael T.~Jury${}^1$}
\address{Department of Mathematics\\
  University of Florida\\
  Box 118105\\
  Gainesville, FL 32611-8105\\
  USA
   }
   \email{mjury@ufl.edu}

\author[McCullough]{Scott McCullough${}^2$}
\address{Department of Mathematics\\
  University of Florida\\
  Box 118105\\
  Gainesville, FL 32611-8105\\
  USA
   }
   \email{sam@ufl.edu}

\thanks{${}^1$ Supported by NSF grant DMS 1101461. ${}^2$ Supported by
  NSF grant DMS 1101137.}


\date{\today}


\maketitle

\begin{abstract}
  It is well known that contractive representations of the disk
  algebra are completely contractive.  The Neil algebra $\cA$ is the
  subalgebra of the disk algebra consisting of those functions $f$ for
  which $f^\prime(0) = 0$.  There is a complete isometry from the
  algebra $R(W)$ of rational functions with poles off of the
  distinguished variety $W = \{(z,w): z^2= w^3, \ \ |z|<1\}$ to $\cA$.
  We prove that there are contractive representations of $\cA$ which
  are not completely contractive, and furthermore provide a Kaiser and
  Varopoulos inspired example whereby $z$ and $w$ in $W$ are
  contractions, yet the resulting representation of $R(W)$ is not
  contractive.  We also present a characterization of those
  contractive representations which are completely contractive.
  Finally, we show that for the variety $\mathcal V=\{(z,w): z^2= w^2, \ \
  |z|<1\}$, all contractive representations of the algebra $R(\mathcal V)$ of
  rational functions with poles off $\mathcal V$ are completely contractive,
  and we as well provide a simplified proof of Agler's analogous
  result over an annulus.
\end{abstract}

\section{Introduction}

Let \df{$\mathbb D$} denote the unit disk in the complex plane and
\df{$\oD$} its closure.  The disk algebra, \df{$\AD$}, is the closure
of analytic polynomials in \df{$C(\oD)$}, the space of continuous
functions on $\oD$ with the supremum norm.  The \df{Neil algebra} is
the subalgebra of the disk algebra given by
\begin{equation*}
  \cA =\{f\in\AD : f^\prime(0)=0\} = \mathbb C + z^2 \AD.
\end{equation*}
Constrained algebras, of which $\cA$ is one of the simplest examples,
are of current interest as a venue for function theoretic operator
theory, such as Pick interpolation.  See for instance
\cite{DPRS,R,JKM,BH} and the references therein.

Let $H$ denote a complex Hilbert space and \df{$B(H)$} the bounded
linear operators on $H$.  A unital representation $\pi:\cA\to B(H)$ on
$H$ is \df{contractive} if $\|\pi(f)\|\le \|f\|$ for all $a\in \cA$,
where $\|f\|$ represents the norm of $f$ as an element of C$(\oD)$ and
$\|\pi(f)\|$ is the operator norm of $\pi(f)$.  Unless otherwise
indicated, in this article representations are unital and contractive.
 
Let $M_n(\cA)$ denote the $n\times n$ matrices with entries from
$\cA$. The norm $\|F\|$ of an element $F=(f_{j,\ell})$ in $M_n(\cA)$
is the supremum of the set $\{\|F(z)\|: z\in \mathbb D\}$, where
$\|F(z)\|$ is the operator norm of the $n\times n$ matrix $F(z)$.
Applying $\pi$ to each entry of $F$,
\begin{equation*}
  \pi^{(n)}(F)= 1_n\otimes\pi(F) = \begin{pmatrix} \pi(f_{j,\ell})
  \end{pmatrix}
\end{equation*}
produces an operator on the Hilbert space $\bigoplus_1^n H$ and
$\|\pi^{(n)}(F)\|$ is then its operator norm. The mapping $\pi$ is
\df{completely contractive} if for each $n$ and $F\in M_n(\cA)$,
\begin{equation*}
  \|\pi^{(n)} (F)\|\le \|F\|.
\end{equation*}

The following theorem is the first  main result of this article.  

\begin{theorem}
 \label{thm:mainneg}
 There exists a finite dimensional Hilbert space and a unital
 contractive representation $\pi:\cA\to B(H)$ which is not completely
 contractive. In fact, there exists a $2\times 2$ matrix rational
 inner function $F$ (with poles outside of the closed disk) such that
 $\|F\|\le 1$, but $\|\pi(F)\|>1$.
\end{theorem}

Theorem \ref{thm:mainpos} gives a necessary and sufficient condition
for a unital representation of $\cA$ to be completely contractive.  An
operator $T\in B(H)$ is a \df{contraction} if it has operator norm
less than or equal to one.  Since the algebra $\cA$ is generated by
$z^2$ and $z^3$, a contractive representation $\pi$ of $\cA$ is
determined by the pair of contractions $X=\pi(z^2)$ and $Y=\pi(z^3)$.
In the spirit of the examples of Kaiser and Varopoulos \cite{KV} for
the polydisk $\mathbb D^d$ ($d>2$), Corollary \ref{cor:noxy} asserts
the existence of commuting contractions $X$ and $Y$ such that
$X^3=Y^2$, but for which the unital representation $\tau$ of $\cA$
with$X = \tau(z^2)$ and $Y = \tau(z^3)$ is not contractive.

Given $0<q<1,$ let $\mathbb A$ denote the annulus $\{z\in\mathbb C:
r<|z|<1\}$ and $A(\mathbb A)$ the annulus algebra, consisting of those
functions continuous on the closure of $\mathbb A$ and analytic in
$\mathbb A$ in the uniform norm. A well known theorem of Agler
\cite{Agler} says that contractive representations of $A(\mathbb A)$
are completely contractive.  If $W$ is a variety in $\mathbb C^2$
which intersects the (topological) boundary of the bidisk $\mathbb
D^2$ only in the torus $\mathbb T^2$, then the set $V= W\cap \mathbb
D^2$ is called a \df{distinguished variety}. 
The annuli (parametrized by $0<q<1$) can be identified with the
distinguished varieties determined by
\begin{equation*}
  z^2 = \frac{w^2-t^2}{1-t^2w^2}
\end{equation*}
for $0<t<1$ \cite{Rudin,Bell,Bell2}.  The limiting case, $z^2=w^2$
corresponds to two disks intersecting at the origin $(0,0)\in\mathbb
C^2$.  Section \ref{sec:ratforA} contains a streamlined proof of
Agler's result which readily extends to show that contractive
representations of the algebra associated to the variety $z^2=w^2$ are
also completely contractive.  See Theorem \ref{thm:V-dilation} and
Corollary \ref{cor:V-dilation}.

The remainder of this introduction places Theorems \ref{thm:mainneg}
and \ref{thm:mainpos} and Corollary \ref{cor:noxy}, as well as Theorem
\ref{thm:V-dilation} and Corollary \ref{cor:V-dilation} in the larger
context of rational dilation.

\subsection{Rational dilation} 

The Sz.-Nagy dilation theorem states that every contraction operator
dilates to a unitary operator.  Unitary operators can be characterized
in various ways, and in particular, they are normal operators with
spectrum contained in the boundary of $\mathbb D$; that is, $\mathbb
T$.  A corollary of the Sz.-Nagy dilation theorem is the von~Neumann
inequality, which implies that $T$ is a contraction if and only if
$\|p(T)\| \leq \|p\|$ for every polynomial $p$, where $\|p\|$ is the
again the norm of $p$ in $C(\oD)$.

More generally, following Arveson \cite{Arveson}, given a compact
subset $X$ of $\mathbb C^d$, let $R(X)$ denote the algebra of rational
functions with poles off $X$ with the norm $\|r\|_X$ equal to the
supremum of the values of $|r(x)|$ for $x\in X$.  The set $X$ is a
\df{spectral set} for the commuting $d$-tuple $T$ of operators on the
Hilbert space $H$ if the spectrum of $T$ lies in $X$ and $\|r(T)\|\le
\|r\|_X$ for each $r\in R(X)$.  If $N$ is also a $d$-tuple of
commuting operators with spectrum in $X$ and acting on the Hilbert
space $K$, then $T$ \df{dilates} to $N$ provided there is an isometry
$V:H\to K$ such that $r(T) = V^* r(N)V$ for all $r\in R(X)$.  The
rational dilation problem asks: if $X$ is a spectral set for $T$ does
$T$ dilate to a tuple $N$ of commuting normal operators with spectrum
in the Shilov boundary of $X$ relative to the algebra $R(X)$?

Choosing $X$ to be the closure of a finitely connected domain $D$ in
$\mathbb C$ with analytic boundary, it turns out the Shilov boundary
is the topological boundary and the problem has a positive answer when
$X$ is an annulus \cite{Agler}.  On the other hand, for planar domains
of higher connectivity, rational dilation fails (at least when the
Schottky double is hyperelliptic (automatic for triply connected
domains), though this is probably an artifact of the proofs and
rational dilation will also fail without this extra condition)
\cite{DMrat,AHR,P}.

With the choice of $X=\oD^d$, the question becomes, if
$T=(T_1,\dots,T_d)$ is a tuple of commuting operators acting on a
Hilbert space $H$ and if
\begin{equation*}
  \|p(T_1,\dots,T_d)\| \le \|p\|_X
\end{equation*}
for every analytic polynomial $p=p(z_1,\dots,z_d)$ in $d$-variables,
does there exist a Hilbert space $K$, an isometry $V:H\to K$, and a
commuting tuple $N=(N_1,\dots,N_d)$ of normal operators on $K$ with
spectrum in $ \mathbb T^d$ (the Shilov boundary of $X$) such that
$p(T) = V^* p(N)V$ for every polynomial $p$?  And\^{o}'s theorem
implies the result is true for the bidisk $\mathbb D^2.$ An example
due to Parrott implies that rational dilation fails for the polydisk
$\mathbb D^d$, $d>2$.  Thus as things stand, the rational dilation
problem has been settled for the disk, the annulus, hyperelliptic
planar domains, and for polydisks.

Arveson \cite{Arveson} gave a profound reformulation of the rational
dilation problem in terms of contractive and completely contractive
representations.
A tuple $T$ acting on the Hilbert space $H$ with spectrum in $X$
determines a unital representation of $\pi_T$ of $R(X)$ on $H$ via
$\pi_T(r) = r(T)$ and the condition that $X$ is a spectral set for $T$
is equivalent to the condition that this representation is
contractive.

Recall that a representation $\pi$ of $R(X)$ is \df{completely
  contractive} if for all $n$ and all $F\in M_n(R(X))$, $\pi^{(n)}(F)
:= (\pi(F_{i,j}))$ is contractive, the norm of $F$ being given by
$\|F\|_\infty =\sup\{ \|F(x)\|: x\in X\}$ with $\|F(x)\|$ the operator
norm of $F(x)$. Arveson showed that $T$ dilates to a tuple $N$ with
spectrum in the (Shilov) boundary of $X$ (with respect to $R(X)$) if
and only if $\pi_T$ is completely contractive. Thus the rational
dilation problem can be reformulated as: Is every contractive
representation of $R(X)$ completely contractive?

The subset $W=\{(z,w)\in\mathbb D^2 : z^2=w^3\}$ of $\mathbb C^2$ 
is a particularly simple but interesting example of a distinguished
variety, called the \df{Neil parabola}.
The mapping from $R(W)$ to the Neil algebra $\cA$ sending $p(z,w)$ to
$p(t^2,t^3)$ is a (complete) isometry.  Much of this paper 
concentrates on studying the connection between contractive and
completely contractive representations of $\cA$, though the results
are readily translated to $R(W)$.  Thus, Theorem \ref{thm:mainneg}
implies that there are contractive representations of $R(W)$ which are
not completely contractive.

Note that excluding a cusp at $(0,0)$, $W$ is a manifold, and this
cusp makes things just different enough so that $R(W)$ a tractable
though nontrivial algebra on which to study the rational dilation
problem. Indeed, many mathematicians have found distinguished
varieties to be attractive venues for function theoretic operator
theory \cite{R,AM05, AM06, AKM, Knese, Vegulla, JKM} and in
particular, they provide interesting examples when trying to delineate the border
between those domains where rational dilation holds and those where it
fails.  Theorem \ref{thm:V-dilation} say that on the distinguished
variety $\mathcal V = \{(z,w)\in\mathbb D^2: z^2=w^2\}$, every contractive
representation of $R(\mathcal V)$ is completely contractive; that is, rational
dilation holds.


While rational dilation fails for the Neil parabola, in Theorem
\ref{thm:mainpos} we also provide a characterization of the completely
contractive representations of $\cA$ \cite{B}.  However, this positive
result is not used to establish Theorem \ref{thm:mainneg}.  Rather the
proof of Theorem \ref{thm:mainneg} essentially comes down to a cone
separation argument.  The mechanics of this argument appear in
Section~\ref{sec:cone-sep}.  The construction of the counterexample
and preliminary results are in Section~\ref{sec:counter}. The proof of
Theorem \ref{thm:mainneg} concludes in Section~\ref{sec:proof}, while
the statement and proof of Theorem \ref{thm:mainpos} and general facts
about representations of $\cA$ are the subject of
Section~\ref{sec:repns}.

The article conclude with Section \ref{sec:ratforA}, which contains a
proof of Agler's rational dilation theorem for the annulus that takes
advantage of subsequent developments in the theory of matrix-valued
functions of positive real part on multiply connected domains.  As a
limiting case, we prove Theorem \ref{thm:V-dilation}, which shows that
rational dilation holds for the algebra $R(\mathcal V)$, $\mathcal V=\{(z,w)\in\mathbb
D^2: z^2=w^2\}$.  Corollary \ref{cor:V-dilation} then gives a
reasonably tractable condition to determine if a given representation
of $R(\mathcal V)$ is contractive, and hence completely contractive.

\section{Representations of $\cA$}
\label{sec:repns}

In this section we characterize the completely contractive
representations of $\cA$ and consider some examples.  The
characterization of contractive representations is essentially
contained in the paper~\cite{DP} on test functions for $\cA$, and this
is described in the next section.

As a (unital) Banach algebra, $\cA$ is generated by the functions
$z^2$ and $z^3$.  It follows that any bounded unital representation is
determined by its values on these two functions.  If $\pi:\cA\to B(H)$
is a bounded representation, $X=\pi(z^2)$ and $Y=\pi(z^3),$ then $X,Y$
are commuting operators which satisfy $X^3=Y^2$.  If we further insist
that $\pi$ is contractive, then $X$ and $Y$ are contractions.  In
summary, every contractive representation $\pi:\cA\to B(H)$ determines
a pair of commuting contractions $X,Y$ such that $X^3=Y^2$.  However,
as we see in Corollary~\ref{cor:noxy}, not every such pair gives rise
to a contractive representation.

The following theorem characterizes the completely contractive
representations of $\cA$.  For Hilbert spaces $H\subset K$, let $P_H$
denote the orthogonal projection of $K$ onto $H$ and $|_H$ the
inclusion of $H$ into $K$.

\begin{theorem}[\cite{B}]
  \label{thm:mainpos}
  A representation $\pi:\cA\to B(H)$ is completely contractive if and
  only if there is a Hilbert space $K\supset H$ and a unitary operator
  $U \in B(K)$ such that for all $n\geq 0$, $n\neq 1$,
  \begin{equation}\label{eq:cc-dil}
    \pi(z^n) = P_H U^n |_H.
  \end{equation}
\end{theorem}

This is a consequence of the Sz.-Nagy dilation theorem together with
applications of the Arveson extension and Stinespring dilation
theorems.  In the case of $\AD$, by the Sz.-Nagy dilation theorem
every completely contractive representation $\pi:\AD\to B(H)$ is
determined by a contraction $T$, with $\pi(z^n)=T^n$, and $T^n=P_H
U^n|_H$ for some unitary $U$ and all $n\geq 0$.  Thus a simple way to
construct completely contractive representations of $\cA$ is to fix a
contraction $T$ and restrict: put $\pi(z^2)=T^2$ and $\pi(z^3)=T^3$.
However, in spite of Theorem~\ref{thm:mainpos} it is {\em not} the
case that every completely contractive representation of $\cA$ arises
in this way, as we see in Example~\ref{eg:no-T} below.

\begin{proof}[Proof of Theorem~\ref{thm:mainpos}]
  Let $\pi:\cA\to B(H)$ be a unital, completely contractive
  representation.  Let $\cA^*\subset C(\mathbb T)$ denote the set of
  complex conjugates of functions in $\cA$.  Then $\cA+\cA^*$ is an
  operator system and $\rho:\cA+\cA^* \to B(H)$ given by
  \begin{equation*}
    \rho(f+g^*) =\pi(f) +\pi(g)^*
  \end{equation*}
  is well defined.  Since $\pi$ is unital and $\cA\cap \cA^*=\mathbb
  C1$, $\rho$ is completely positive.  By the Arveson extension
  theorem, $\rho$ extends to a unital, completely positive (ucp) map
  $\sigma:C(\mathbb T)\to B(H)$.  By the Stinespring theorem there is
  a larger Hilbert space $K\supset H$, and a unitary $U\in B(K)$ such
  that for all $n\geq 0$,
  \begin{equation*}
    \sigma(z^n) =P_HU^n|_H.
  \end{equation*}
  Since $\pi(z^n)=\sigma(z^n)$ for all nonnegative $n\neq 1$, one
  direction follows.

  Conversely, suppose that there is a unitary operator $U \in B(K)$
  such that for all $n\geq 0$, $n\neq 1$, $\pi(z^n) = P_H U^n |_H$.
  Then $\tilde\pi$ defined as $\tilde\pi(z^n) = U^n$, $n \in \mathbb
  Z$ defines a completely contractive representation of $C(\mathbb
  T)$.  So $\tilde\pi$ restricted the operator system $\cA\cap \cA^*$
  is completely positive, as is $\rho$, its compression to $H$, by the
  Stinespring dilation theorem.  Since unital completely positive maps
  are completely contractive, $\pi=\rho | \cA$ is completely
  contractive.
\end{proof}

\begin{remark}
  In the above proof, obviously $T=P_HU|_H$ is a contraction.  However
  since the restriction of $\sigma$ to $\AD$ is not necessarily
  multiplicative, we cannot conclude that $\pi(z^2)=T^2$ and
  $\pi(z^3)=T^3$.  Indeed the following example illustrates this
  concretely:
\end{remark}

\begin{example}
  \label{eg:no-T}
  Let $K$ be a separable Hilbert space with orthonormal basis
  $\{e_j\}_{j\in\mathbb Z}$, and let $U$ be the bilateral shift.  Let
  $H \subset K$ be defined as $H = e_0 \vee \bigvee_{n=2}^\infty e_n$.
  Then $H$ is invariant for $U^2$ and $U^3$, and so by
  Theorem~\ref{thm:mainpos}, $\pi$ given by $\pi(z^n) = P_H U^n |_H
  = U^n |_H$, $n\geq 0$, $n\neq 1$, is a completely contractive
  representation of $\cA$.

  If it were the case that for some $T\in B(H)$, $T^2 = \pi(z^2)$ and
  $T^3 = \pi(z^3)$, we would require that
  \begin{equation*}
    e_3 = U^3e_0 = \pi(z^3) = \pi(z^2) Te_0.
  \end{equation*}
  However, $\ip{\pi(z^2)e_n}{e_3} = \ip{U^2 e_n}{e_3} = 0$ for $n\geq
  0$, $n\neq 1$, and hence $e_3$ is orthogonal to the range of
  $\pi(z^2)$.  Thus there is no way to define $Te_0$ so that $e_3 =
  \pi(z^2) Te_0$, and so there can be no such $T$.
\end{example}

 \begin{example}
 \label{cantuseando}
 If $\pi:\cA\to B(H)$ is a unital contractive representation, then the
 image of the generators $z^2,z^3$ of $\cA$ are evidently
 contractions, $S=\pi(z^2)$ and $T=\pi(z^3)$.  Further $S^3=T^2$.  By
 Ando's Theorem, there exists a pair of commuting unitaries $X$ and
 $Y$ on a larger Hilbert space $K$ containing $H$ such that
 \begin{equation*}
   S^n T^m = V^* X^n Y^m V,
 \end{equation*}
 where $V$ is the inclusion of $H$ into $K$.  Because $X$ and $Y$ are
 unitary and commute, $X^*Y=YX^*$ by the Putnam-Fuglede theorem.  The
 operator $U=X^*Y$ is a contraction, but unfortunately, there is no
 reason to expect that $U^2 = X$ and $U^3=Y$ or equivalently,
 $X^3=Y^2$.  In general then, it will not be the case that $V^*
 U^{2n+3m}V = S^nT^m=\pi(z^{2n+3m})$.  Indeed, Theorems
 \ref{thm:mainneg} and Theorem \ref{thm:mainpos} imply that $\pi$
 contractive is not a sufficient assumption to guarantee the existence
 of such a $U$.

 It is worth noting that the construction of $U=X^*Y$ via Ando's
 Theorem did not use the full strength of the contractive hypothesis
 on $\pi$, but rather only that $S$ and $T$ are commuting contractions
 with $S^3=T^2$.  Perhaps surprisingly, in view of Corollary
 \ref{cor:noxy} below, the representation $\pi$ of $\cA$ determined by
 $\pi(z^2)=S$ and $\pi(z^3)=T$ need not even be contractive.
\end{example}

\section{The set of test functions and its cone}
\label{sec:cone-sep}

Given $\lambda\in\mathbb D$, let \index{$\lambda$}
\begin{equation}
  \label{eq:varphi}
  \varphi_\lambda(z) = \frac{z-\lambda}{1-\lambda^* z},
\end{equation}
and let 
\begin{equation}
  \label{eq:phi}
  \psi_\lambda(z) = z^2  \varphi_\lambda(z)
\end{equation}
the (up to a unimodular constant) \df{Blaschke factor}
\index{$\psi_\lambda$} with zero at $\lambda$, times $z^2$.  It will
be convenient to let
\begin{equation*}
  \psi_\infty = z^2
\end{equation*}
and \index{$\psi_\infty$} at the same time let $\infty$ denote the
point at infinity in the one point compactification \df{$\Dinf$} of
the unit disk $\mathbb D$.  Let \index{$\Psi$}
\begin{equation*}
  \Psi = \{\psi_\lambda: \lambda \in\Dinf \},
\end{equation*}
 with the topology and Borel structure inherited from
$\Dinf$.  We refer to this as a set of \df{test functions}.  It has
the properties that it separates the points of $\mathbb D$ and for all
$z\in \mathbb D$, $\sup_{\psi\in\Psi} |\psi(z)| < 1$.

Recall that for a set $X$ and $C^*$-algebra $\mathcal A$, a function
$k: X\times X \to \mathcal A$ is called a \df{kernel}.  It is a
\df{positive kernel} if for every finite subset $\{x_1, \dots,x_n\}$
of $X$, $(k(x_i,x_j)) \in M_n(\mathcal A)$ is positive semidefinite.

Let $M(\Psi)$ be the space of finite Borel measures on the set of test
functions. Given a subset $S$ of $\mathbb D$,  denote by 
 $M^+(S) = \{\mu:S\times S\to
M(\Psi)\}$ the collection of positive kernels on $S\times S$ 
 into $M(\Psi)$.  Write $\mu_{xy}$ for the value of
$\mu$ at the pair $(x,y)$.  By $\mu$ being positive, we mean that for
all finite sets $\mathcal G\subset S$ and all Borel sets
$\omega\subset \Psi$, the matrix
\begin{equation}
(\mu_{x,y}(\omega) )_{x,y\in\mathcal G}  
\end{equation}
is positive semidefinite.  For example, if $\mu$ is identically equal
to a fixed positive measure $\nu$, or more generally is of the form
$\mu_{xy}=f(x)f(y)^*\nu$ for a fixed positive measure $\nu$ and
bounded measurable function $f:\mathbb C\to \mathbb D$, or more
generally still is a finite sum of such terms, then it is positive.

Our starting point is the following result from \cite{DP} (stated
there for functions of positive real part):

\begin{proposition}\label{prop:scalar-cone}
  An analytic function $f$ in the disk belongs to $\cA$ and satisfies
  $\|f\|_\infty \leq 1$ if and only if there is a positive kernel
  $\mu\in M^+(\mathbb D)$ such that
  \begin{equation}
    1-f(x)f(y)^* =\int_\Psi (1-\psi(x)\psi(y)^*)\, d\mu_{xy}(\psi).
  \end{equation}
  for all $x,y\in \mathbb D$.  Furthermore, $\Psi$ is minimal, in the
  sense that there is no proper closed subset of $E \subset \Psi$ such
  that for each such $f$, there exists a $\mu$ such that 
  \begin{equation}
    1-f(x)f(y)^* =\int_E (1-\psi(x)\psi(y)^*)\, d\mu_{xy}(\psi).
  \end{equation}
\end{proposition}

For $E\subset \Psi$ a closed subset, let $C_{1,E}$ denote the cone
consisting of the kernels
\begin{equation}
  \left(\int_E (1-\psi(x)\psi(y)^*)\, d\mu_{x,y}(\psi)
  \right)_{x,y\in\mathbb D}.
\end{equation}
(Equivalently, we could consider only those $\mu$ such that $\mu_{xy}$
is supported in $E$ for all $x,y$.)  In particular, if we choose $E =
\{z^2,z^3\}$, it follows from \cite[Theorem 3.8]{DP} that there exists
a function $f\in \cA$ with $\|f\|_\infty \leq 1$ such that
$1-f(x)f(y)^*\notin C_{1,E}$.  This yields in our context an analogue
of the Kaiser and Varopoulos example for the tridisk:

\begin{corollary}
 \label{cor:noxy}
  There exists a pair of commuting contractive matrices $X,Y$ with
  $X^3=Y^2$, but such that the representation of $\cA$ determined by
  $\pi(z^2)=X$, $\pi(z^3)=Y$ is not contractive.
\end{corollary} 

\begin{proof}
  By a cone separation argument as in the proof of
  Proposition~\ref{prop:counterifsep}, there is a bounded
  representation $\pi$ of $\cA$ (determined by a pair of matrices
  $X,Y$ with spectrum in $\mathbb D$) such that $\|\pi(\psi)\| \leq 1$
  for each $\psi\in E$ but $\|\pi(f)\| > 1$.  In particular, if we
  take $E$ to be the closed set $\{z^2, z^3\}$, we see that
  $X=\pi(z^2)$ and $Y=\pi(z^3)$ satisfy the conditions of the
  corollary.
\end{proof}

\subsection{The matrix cone}
 \label{sec:matcone}
 To study the action of representations on $M_2(\cA)$, consider a finite
subset $\fsg\subset \mathbb D.$ \index{$\fsg$}
As usual, $M_2(\mathbb C)$ stands for the
$2\times 2$ matrices with entries from $\mathbb C$.  Let $\cX_{2,\fsg}$
denote the set of all kernels $G:\fsg\times\fsg\to M_2(\mathbb C)$ and
$\cL_{2,\fsg}\subset \cX_{\fsg}$ denote the selfadjoint kernels
$F:\fsg\times \fsg \mapsto M_2(\mathbb C)$, in the sense that $F(x,y)^*
= F(y,x)$.  Finally, write \df{$C_{2,\fsg}$} for the cone in
$\cL_{2,\fsg}$ of elements of the form
\begin{equation}
 \label{eq:inCfs}
  \begin{pmatrix} \int_{\Psi} (1-\psi(x)\psi(y)^*)\, d\mu_{x,y}(\psi)
  \end{pmatrix}_{x,y\in \fsg}
\end{equation}
where $\mu = (\mu_{x,y}) \in M_2^+(\fsg)$ is a kernel taking its values
$\mu_{x,y}$ in the $2\times 2$ matrix valued measure on $\Psi$ such
that the measure
\begin{equation}
 \label{eq:inCfsM}
  M(\omega) = \begin{pmatrix} \mu_{x,y}(\omega) \end{pmatrix}_{x,y}
\end{equation}
takes positive semidefinite values (in $M_\fn(M_2(\mathbb C))$).  Given
$f:\fsg \to \mathbb C^2$, the kernel $(f(x)f(y)^*)_{x,y\in \fsg}$ is
called a \df{square}.

\begin{lemma}
  The cone $C_{2,\fsg}$ is closed and contains all squares. 
\end{lemma}

\begin{proof}
  For $x \in\fsg$,
  \begin{equation*}
    \sup_{\psi\in\Psi} | \psi(x)| < |x|.
  \end{equation*}
  Hence as $\fsg$ is finite, there exists a there exists $0<\kappa
  \leq 1 $ such that for all $x\in\fsg$ and $\psi \in\Psi$
  \begin{equation*}
    1-\psi(x)\psi(x)^* \ge \kappa.
  \end{equation*}
  Consequently, if $\Gamma$ defined by
  \begin{equation*}
    \Gamma(x,y) = \int_\Psi (1-\psi(x)\psi(y)^*)\, d\mu_{x,y}(\psi) 
  \end{equation*}
  is in $C_{2,\fsg}$, then 
  \begin{equation*}
    \frac{1}{\kappa}  \Gamma(x,x) \succeq \mu_{x,x}(\Psi),
  \end{equation*}
  where the inequality is in the sense of the positive semidefinite
  order on $2\times 2$ matrices.

  Now suppose $(\Gamma_n)$ is a sequence from $C_{2,\fsg}$ converging
  to some $\Gamma$.  For each $n$ there is a measure $\mu^n$ such that
  $\Gamma_n$ given by
  \begin{equation*}
    \Gamma_n(x,y)  = \int_\Psi (1-\psi(x)\psi(y)^*)\,
    d\mu^n_{x,y}(\psi)
  \end{equation*} 
  forms a sequence from $C_{2,\fsg}$ which converges to some $\Gamma$.
  Hence there exists a $\tilde\kappa > 0$ such that for all $n$ and
  all $x\in\fsg$, $\tilde\kappa \geq \Gamma_n(x,x)$.  Consequently, for
  all $n$ and all $x\in\fsg$,
  \begin{equation*}
    \tfrac{\tilde\kappa}{\kappa} I \succeq \mu^n_{x,x}.
  \end{equation*}
  By positivity of the $\mu^n$s, it now follows that the measures
  $\mu^n_{x,y}$ are uniformly bounded.  Hence there exists a
  subsequence $\mu^{n_j}$ and a measure $\mu$ such that $\mu^{n_j}$
  converges weak-$*$ to $\mu$, which therefore is positive.  We
  conclude that
  \begin{equation*}
    \Gamma = \int_\Psi (1-\psi(x)\psi(y)^*)\,d \mu_{x,y}(\psi) \in
    C_{2,\fsg},
  \end{equation*}
  establishing the fact that $C_{2,\fsg}$ is closed.
  
  Now let $f:\fsg\to\mathbb C^2$ be given.  Let $\delta$ denote the
  unit scalar point mass \index{$\delta$} \index{point mass} at
  $z^3$ ($\lambda=0$).  Then for $\omega\subset\Psi$ a Borel subset,
  \begin{equation*}
    \mu_{x,y}(\omega) = f(x)\frac{1}{1-x^3 y^{*3}} \delta(\omega)
    f(y)^*
  \end{equation*}
  defines a positive measure and
  \begin{equation*}
    \int_\Psi (1- \psi(x)\psi(y)^*) \, d\mu_{x,y}(\psi)  = f(x)
    f(y)^*,
  \end{equation*}
  showing that $C_{2,\fsg}$ contains the squares.
\end{proof}

Elaborating on the construction at the end of the last proof, if
\begin{equation*}
  \nu(\omega) = \begin{pmatrix} \nu_{x,y}(\omega) \end{pmatrix}_{x,y
    \in \fsg}
\end{equation*}
is positive semidefinite for every Borel subset $\omega$ of $\Psi$,
each $\nu_{xy}$ a scalar valued measure, and if $f:\fsg \to \mathbb
C^2$, then
\begin{equation*}
 \mu_{x,y}(\omega) = f(x) \nu_{x,y}(\omega) f(y)^*,
\end{equation*}
defines an $M_2(\mathbb C)$ valued positive measure $\mu$ and
\begin{equation*}
  \int_{\Psi} (1-\psi(x)\psi(y)^*) \, d\mu_{x,y}(\psi) \in C_{2,\fsg}.
\end{equation*}
We therefore have the following from \cite{DP} (see also \cite{BBtH}).

\begin{proposition}
 \label{prop:mad+}
  If $g\in \cA$ is analytic in a neighborhood
  of the closure of the disk and if $\|g\|_\infty \le 1$, then 
  $1-g(x)g(y)^* \in C_{2,\fsg}(1)$.  Thus, if $f:\fsg\to \mathbb C^2$,
 then 
\begin{equation*}
  f(x)(1-g(x)g(y)^*)f(y)^* \in C_{2,\fsg}.
\end{equation*}
\end{proposition}


\subsection{The cone separation argument}
Continue to let $\fsg$ denote a finite subset of $\mathbb D.$
Given $F\in M_2(\cA)$, let $\Sigma_{F,\fsg}$ \index{$\Sigma_{F,\fsg}$}
denote the kernel
\begin{equation}
 \label{eq:SigmaF}
 \Sigma_{F,\fsg} = (1-F(x)F(y)^*)_{x,y\in\fsg}. 
\end{equation}
Let \df{$I$} denote the ideal of functions in $\cA$ which vanish on
$\fsg$.  Write $q:\cA\to \cA/I$ for the canonical projection, which is
completely contractive.  We use the standard notation \df{$\sigma(T)$}
for the spectrum of an operator $T$ on Hilbert space, as well as $F^t$
for the transpose of the matrix function $F$.  Thus, $F^t(z)=F(z)^t$.
Obviously, when $F \in M_2(\cA)$, $F^t$ is as well, and $\|F\|_\infty
= \|F^t\|_\infty$.

\begin{proposition}
 \label{prop:counterifsep}
 If $F\in M_2(\cA)$, but $\Sigma_{F,\fsg}\notin C_{2,\fsg}$, then
 there exists a a Hilbert space $H$ and representation $\tau:\cA/I \to
 B(H)$ such that
 \begin{enumerate}[(i)]
 \item  $\sigma(\tau(a))\subset a(\fsg)$ for $a\in\cA;$ 
 \item $\| \tau(q(a)) \| \le 1$ for all $a\in \cA$ with $\|a\| \leq
   1$; but
 \item $\|\tau^{(2)}(q(F^t)) \| > 1$.
 \end{enumerate}
 Therefore if $\|F\| \le 1$, then the representation $\tau\circ q$ is
 contractive, but not completely contractive.
\end{proposition}

\begin{proof}
  The proof proceeds by a cone separation argument: the representation
  is obtained by applying the GNS construction to a linear functional
  that separates $\Sigma_{F,\fsg}$ from $C_{2,\fsg}$.

  The cone $C_{2,\fsg}$ is closed and by assumption $\Sigma_{F,\fsg}$
  is not in the cone.  Hence there is an $\mathbb R$-linear functional
  $\Lambda : \cL_\fsg \to \mathbb R$ such that $\Lambda(C_{2,\fsg})
  \ge 0$, but $\Lambda (\Sigma_{F,\fsg})<0$.  Given $f:\fsg\to \mathbb
  C^2$ (that is, $f \in (\mathbb C^2)^\fsg$), recall that the square
  $ff^*:=(f(x)f(y)^*)_{x,y\in\fsg}$ is in the cone and hence
  $\Lambda(ff^*)\ge 0$.  Since every element of $\cX_\fsg$ can be
  expressed uniquely in the form $G=U+iV$ where $U,V\in \cL_\fsg$,
  there is a unique extension of $\Lambda$ to a $\mathbb C$-linear
  functional $\Lambda :\cX_\fsg\to \mathbb C$.  With this extended
  $\Lambda$, let $H$ denote the Hilbert space obtained by giving
  $(\mathbb C^2)^\fsg$ the (pre)-inner product
  \begin{equation*}
    \langle f,g\rangle = \Lambda(fg^*)
  \end{equation*}
  and passing to the quotient by the space of null vectors (those $f$
  for which $\Lambda(ff^*)=0$ --- since $\fsg$ is finite, the quotient
  will be complete).

  Define a representation $\rho$ of $\cA$ on $H$ by
  \begin{equation*}
    \rho(g) f(x) =g(x)f(x),
  \end{equation*}
  where the scalar valued $g$ multiplies the vector valued $f$
  entrywise.

  If $g\in \cA,$ is analytic in a neighborhood of the closure of the
  disk and $\|g\|_\infty \le 1$, then, by Proposition \ref{prop:mad+},
  $f(x)(1-g(x)g(y)^*)f(y) \in C_{2,\fsg}$.  Thus,
  \begin{equation}
    \label{eqn:rho-contractive}
    \ip{f}{f} - \ip{\rho(g)f}{\rho(g) f}
    = \Lambda\left(( f(x)(1-g(x)g(y)^*)f(y)^*)_{x,y\in\fsg}\right)
    \ge 0.
  \end{equation}
  Hence, if $\|g\|_\infty \le 1$, then $\|\rho(g)\| \le 1$ and $\rho$
  is a contractive representation of $\cA$.  Moreover, since the
  definition of $\rho$ depends only on the values of $g$ on $\mathcal
  F$, it passes to a contractive representation $\tau:\cA/I\to B(H)$.
  The restriction of $\cA$ to $\mathcal F$ separates points of
  $\mathcal F$ (indeed, the elements of $\Psi$ do so), and so it
  follows that for each $a\in\cA$ the eigenvalues of the matrix
  representing $\tau(a)$ constitute the set $a(\mathcal F)$.  This
  proves (\textit{i}) and (\textit{ii}).

  To prove (\textit{iii}), let $\{e_1,e_2\}$ denote the standard
  basis for $\mathbb C^2$ and let $[e_j]:\fsg \to \mathbb C^2$ be the
  constant function $[e_j](x) = e_j$.  Note that
  $\{e_ie_j^*\}_{i,j=1}^2$ are a system of $2\times 2$ matrix units.
  We find
  \begin{equation*}
    \rho^{(2)}(F^t) ([e_1]\oplus [e_2])
   =  \begin{pmatrix} F_{1,1}e_1 + F_{2,1}e_2 \\ F_{1,2}e_1 +
     F_{2,2}e_2\end{pmatrix}.
 \end{equation*}
 Since
 \begin{equation*}
   \begin{split}
     (F_{1,1}e_1 + F_{2,1}e_2)(F_{1,1}e_1 + F_{2,1}e_2)^* &=
     F_{1,1}F_{1,1}^*e_1e_1^* + F_{2,1}F_{1,1}^*e_2e_1^* +
     F_{1,1}F_{2,1}^*e_1e_2^* + F_{2,1}F_{2,1}^* e_2e_2^* \\
     & =
     \begin{pmatrix}
       F_{1,1}F_{1,1}^* & F_{1,1}F_{2,1}^* \\
       F_{2,1}F_{1,1}^* & F_{2,1}F_{2,1}^*
     \end{pmatrix},
\end{split}
\end{equation*}
 and 
\begin{equation*}
 \begin{split}
     (F_{1,2}e_1 + F_{2,2}e_2)(F_{1,2}e_1 + F_{2,2}e_2)^* &=
     F_{1,2}F_{1,2}^*e_1e_1^* + F_{2,2}F_{1,2}^*e_2e_1^* +
     F_{1,2}F_{2,2}^*e_1e_2^* + F_{2,2}F_{2,2}^* e_2e_2^* \\
     & =
     \begin{pmatrix}
       F_{1,1}F_{1,1}^* & F_{1,1}F_{2,1}^* \\
       F_{2,1}F_{1,1}^* & F_{2,1}F_{2,1}^*
     \end{pmatrix},
   \end{split}
 \end{equation*}
 it follows that 
 \begin{equation*}
   \begin{split}
     \ip{\rho^{(2)}(F^t) ([e_1]\oplus [e_2])}{\rho^{(2)}(F^t)
       ([e_1]\oplus [e_2])} &= \Lambda \left(
       \begin{pmatrix}
         F_{1,1}F_{1,1}^* + F_{1,2}F_{1,2}^* && F_{1,1}F_{2,1}^* +
         F_{1,2}F_{2,2}^* \\ 
         F_{2,1}F_{1,1}^* + F_{2,2}F_{1,2}^* && F_{2,1}F_{2,1}^* +
         F_{2,2}F_{2,2}^*
       \end{pmatrix}
     \right) \\
     &= \Lambda(F F^*),
   \end{split}
 \end{equation*}
 and so
 \begin{equation*}
   \ip{(I-\rho^{(2)}(F^t)^*\rho^{(2)}(F^t)) [e_1]\oplus
     [e_2]}{[e_1]\oplus [e_2]} < 0.
 \end{equation*}
 We conclude that $\|\rho(F^t)\| > 1$, and in particular, if it
 happens to be the case that $\|F\|_\infty \leq 1$, then $\rho$ is not
 $2$-contractive, and thus not completely contractive.
\end{proof}

\begin{remark}
  \label{rem:SigmaFandT}
  Though it is not needed in what follows, observe that the converse
  of the first part of Proposition \ref{prop:counterifsep} is true: If
  $T$ is an operator on Hilbert space with spectrum in $\fsg$, if
  $\Sigma_{F,\fsg} \in C_{2,\fsg}$ and if $\psi(T)$ is contractive for
  all $\psi \in \Psi$, then $F(T)$ is also contractive.

  A proof follows along now standard lines (see, for instance,
  \cite{DM}, where the needed theorems are proved for scalar valued
  functions, though the proofs remain valid in the matrix case).  The
  assumption that $\Sigma_{F,\fsg} \in C_{2,\fsg}$ means that $F$ has a
  $\Psi$-unitary colligation transfer function representation.  Since
  the operator $T$ has spectrum in the finite set $\fsg$, it determines
  a representation of $\cA$ which sends bounded pointwise convergent
  sequences in $M_2(\cA)$ to weak operator topology convergent
  sequences.  Representations of $M_2(\cA)$ with this property and for
  which $\psi(T)$ is contractive for all $\psi \in \Psi$, are
  contractive.
\end{remark}

\section{Construction of the counterexample preliminaries}
 \label{sec:counter}

For $\lambda\in\mathbb D\backslash \{0\}$, let
\begin{equation*}
 \varphi_\lambda = \frac{z-\lambda}{1-\lambda^*z}.
\end{equation*}
Fix distinct points $\lambda_1,\lambda_2\in\mathbb D.$  As a
shorthand notation, write $\varphi_j=\varphi_{\lambda_j}$.
\index{$\Phi$} Set
\begin{equation}
  \label{eq:Phi}
  \Phi = \frac{1}{\sqrt{2}}
  \begin{pmatrix} \varphi_1 & 0 \\ 0 & 1 \end{pmatrix} 
  U
  \begin{pmatrix} 1 & 0 \\ 0 & \varphi_2\end{pmatrix},
\end{equation}
where $U$ is a $2\times 2$ unitary matrix with no non-zero entries.
To be concrete, choose
\begin{equation*}
  U = \begin{pmatrix} 1 & 1 \\ 1 & -1 \end{pmatrix}.
\end{equation*}
In particular $\Phi$ is a $2\times 2$ matrix inner function with $\det
\Phi(\lambda)=0$ at precisely the two nonzero points $\lambda_1$ and
$\lambda_2$.  The function
\begin{equation}
 \label{eq:bF}
  F = z^2\Phi
\end{equation}
is in $M_2(\cA)$ and is a rational inner function, so $\|F\|_\infty =
1$.

Ultimately we will identify a finite set $\fs$ and show that
$\Sigma_{F,\fs} \not\in \cC_{2,\fs}$ and thus, in view of
Proposition~\ref{prop:counterifsep} establish Theorem
\ref{thm:mainneg}.  In the remainder of this section we collect some
needed preliminary lemmas.

\begin{lemma}
  \label{lem:badPhi}
  Given distinct points $\lambda_1,\lambda_2 \in \mathbb
  D\backslash\{0\}$ and a $2\times 2$ unitary matrix $U$, let
  \begin{equation}
    \label{eq:badPhi}
    \Theta = \begin{pmatrix} \varphi_1 & 0 \\ 0 & 1 \end{pmatrix} U
    \begin{pmatrix} 1 & 0 \\ 0 & \varphi_2 \end{pmatrix},
  \end{equation}
  where $\varphi_j = \varphi_{\lambda_j}$.  The matrix $U$ is diagonal;
  that is, there exist unimodular constants $s$ and $t$ such that
  \begin{equation*}
    \Theta = \begin{pmatrix} s \varphi_1 & 0 \\ 0 & t\varphi_2
    \end{pmatrix},
  \end{equation*}
  if and only if there exists points $a,b\in\mathbb D$ and $2\times 2$
  unitaries $V$ and $W$ such that
  \begin{equation*}
    \Theta = V^* \begin{pmatrix} \varphi_a & 0 \\ 0 & \varphi_b
    \end{pmatrix} W.
  \end{equation*}
\end{lemma}

\begin{proof}
  The forward implication is trivial.  For the converse, let
  $\{e_1,e_2\}$ denote the standard basis for $\mathbb R^2$.  By
  taking determinants, it follows that $\{a,b\} = \{\lambda_1,
  \lambda_2\}$.  Changing $V$ and $W$ if necessary, without loss of
  generality it can be assumed that $a=\lambda_1$ and $b=\lambda_2$.
  Evaluating at $\lambda_2$ it follows that $We_2 = \alpha e_2.$
  Because $W$ is unitary, it now follows that $W$ is diagonal.  A
  similar argument shows that $V$ is diagonal, and the result follows.
\end{proof}

\begin{lemma}
  \label{lem:posI}
  Suppose $\mu_{i,j}$ are $2\times 2$ matrix-valued measures on a
  measure space $(X,\Sigma)$ for $i,j=0,1$.  If $\mu_{i,j}(X)=I$ for
  all $i,j$ and if, for each $\omega\in\Sigma$ the $4\times 4$
  $($block $2\times 2$ matrix with $2\times 2$ matrix entries$)$
  \begin{equation*}
    \begin{pmatrix} \mu_{i,j}(\omega)\end{pmatrix}_{i,j=1}^2
  \end{equation*}
  is positive semidefinite, then $\mu_{i,j}=\mu_{0,0}$ for each
  $i,j=0,1$.
\end{lemma}

\begin{proof}
  Fix a unit vector $f\in\mathbb C^2$ and let
  \begin{equation*}
    \nu_{i,j}(\omega) = \ip{\mu_{i,j}(\omega)f}{f}.
  \end{equation*}
  It follows that $\nu_{i,j}(X) = 1$ and for each $\omega \in \Sigma$
  \begin{equation*}
    \gamma(\omega)  = \begin{pmatrix} \nu_{i,j}(\omega)
    \end{pmatrix}_{i,j=1}^2
  \end{equation*}
  is positive semidefinite.  On the other hand,
  \begin{equation*}
    \gamma(X) -\gamma(\omega)\ge 0
  \end{equation*}
  and since $\gamma(X)$ is rank one (with a one in each entry), there
  is a constant $c=c_\omega$ such that
  \begin{equation*}
    \gamma(\omega) = c\gamma(X).
  \end{equation*}
  Consequently, $\nu_{i,j}(\omega) = \nu_{1,1}(\omega)$.  By
  polarization it now follows that $\mu_{i,j} = \mu_{1,1}$ for each
  $i,j=1,2.$
\end{proof}

\begin{lemma}
  \label{lem:I-PP*}
  There exist independent vectors $v_1,v_2\in\mathbb C^2$ and, for any
  finite subset $\fsg$ of the disc, functions $a,b:\fsg\to\mathbb C^2$
  in the span of $\{x^2 k_{\lambda_1}(x)v_1, x^2k_{\lambda_2}(x)v_2\}$
  such that
  \begin{equation*}
    \frac{I-\Phi(x)\Phi(y)^*}{1-xy^*} = a(x)a(y)^* + b(x)b(y)^*.
  \end{equation*}
\end{lemma}

\begin{proof}
  Let $M_\Phi$ denote the operator of multiplication by $\Phi$ on
  $H^2_{\mathbb C^2}$, the Hardy-Hilbert space of $\mathbb C^2$-valued
  functions on the disk.  Because $\Phi$ is unitary-valued on the
  boundary, $M_\Phi$ is an isometry.  In fact, $M_\Phi$ is the product
  of three isometries in view of Equation \eqref{eq:Phi}.  The
  adjoints of the first and third have one dimensional kernels.  The
  middle term is unitary and so its adjoint has no kernel.  Thus, the
  kernel of $M_\Phi^*$ has dimension at most two.  It is evident that
  $k_{\lambda_1}e_1$ is in the kernel of $M_\Phi^*$.  Choose a unit
  vector $v_2$ in $\mathbb C^2$ with entries $\alpha$ and $\beta \neq
  0$ such that
  \begin{equation*}
    \begin{pmatrix} \alpha \varphi_{\lambda_1}(\lambda_2) \\ \beta
    \end{pmatrix} = U e_2,
  \end{equation*}
  with $U$ the unitary appearing in Equation \eqref{eq:Phi}.  That
  such a choice of $\alpha$ and $\beta\ne 0$ is possible follows from
  the assumption that $\lambda_1\neq\lambda_2$, which ensures that
  $\varphi_{\lambda_1}(\lambda_2)\ne 0$, and the assumption that $U$
  has no non-zero entries, giving $\beta\neq 0$.  Further, with
  this choice of $v_2$ a simple calculation shows that
  $k_{\lambda_2}v_2$ is also in the kernel of $M_\Phi^*$.  Hence, the
  dimension of the kernel of $M_\Phi^*$ is two.  Since $M_\Phi$ is an
  isometry, $I-M_\Phi M_\Phi^*$ is the projection onto the kernel of
  $M_\Phi^*$.

  Choose an orthonormal basis $\{a,b\}$ for the kernel of $M_\Phi^*$
  so that $ I-M_\Phi M_\Phi^* = aa^* + bb^*.$ It now follows that, for
  vectors $v,w\in\mathbb C^2$,
  \begin{equation*}
    \begin{split}
      \ip{\frac{I-\Phi(x)\Phi(y)^*}{1-xy^*} v}{w}
   = &\, \ip{(I-M_\Phi M_\Phi^*) k_yv}{k_xw} \\
   = &\, \ip{(aa^*+bb^*) k_y v}{k_x w} \\
   = &\, \ip{ k_yv}{a}\, \ip{a}{k_xw} + \ip{k_yv}{b}\, \ip{b}{k_xw} \\
   = &\, \ip{(a(x)a(y)^* + b(x)b(y)^*)v}{w}.   
 \end{split}
\end{equation*}
\end{proof}

The following is well known. 

\begin{lemma}
  \label{lem:szegoinv}
  Let $s$ be the Szeg\H{o} kernel,
  \begin{equation*}
    s(x,y) =\frac{1}{1-xy^*}.
  \end{equation*}
  If $x_1,\dots,x_m$ and $y_1,\dots,y_m$ are two $m$-tuples each of
  distinct points in the unit disk $\mathbb D,$ then the matrix
  \begin{equation*}
    M=\begin{pmatrix}  s(x_j,y_\ell) \end{pmatrix}_{j,\ell=1}^n
  \end{equation*} 
  is invertible.
\end{lemma}

\begin{proof}
  Suppose $Mc=0$ where $c$ is the vector with entries $c_1,\dots,c_m.$
  Let
  \begin{equation*}
    r(x) = \sum c_\ell  s(x,y_\ell) =  [(1-xy_1^*)\cdots (1-xy_m^*)]^{-1}
    \sum c_\ell p_\ell(x),
  \end{equation*}
  for polynomials $p_\ell$ of degree $m-1$.  Hence $r$ is a rational
  function with numerator a polynomial $p$ of degree at most $m-1$ and
  denominator which does not vanish on $\mathbb D$.  The hypotheses
  imply that $p(x_j)=0$ for $j=1,2,\dots,m.$ Hence $p$ is identically
  zero, as then is $r$.  Since the kernel functions
  $\{s(\cdot,t_\ell): \ell =1,2,\dots,m\}$ form a linearly independent
  set in $H^2(\mathbb D)$, it follows that $c=0$.
\end{proof}

Given a $2\times 2$ matrix valued measure and a vector
$\gamma\in\mathbb C^2$, let $\nu_\gamma$ denote the scalar measure
defined by $\nu_\gamma(\omega)= \gamma^* \nu(\omega)\gamma$.  Note
that if $\nu$ is a positive measure (that is, takes positive
semidefinite values), then each $\nu_\gamma$ is a positive measure.
 Let $\Psi_0=\Psi\backslash\{\psi_\infty\}$. \index{$\Psi_0$}

\begin{lemma}
  \label{lem:finiteistwo}
  Suppose $\nu$ is a $2\times 2$ positive matrix-valued measure on
  $\Psi_0.$   For each $\gamma$ the
  measure $\nu_\gamma$ is a nonnegative linear combination of at most
  two point masses if and only if there exist $($possibly not
  distinct$)$ points $\fz_1,\fz_2$ and positive semidefinite matrices
  $Q_1$ and $Q_2$ such that \index{$Q$}
  \begin{equation*}
    \nu = \sum_{j=1}^2 \delta_{\fz_j} Q_j,
  \end{equation*}
  where $\delta_{\fz_1}, \delta_{\fz_2}$ are scalar unit point
  measures on $\Psi$ supported at $\psi_{\fz_1}, \psi_{\fz_2}$,
  respectively.
\end{lemma}

\begin{proof}
  If $\nu = \sum_{j=1}^2 \delta_{\fz_j} Q_j$ with $\fz_1,\fz_2$ and
  $Q_1, Q_2$ as in the statement of the lemma, then clearly each
  $\nu_\gamma$ is a nonnegative linear combination of at most two
  point masses.

  For the converse, the $M_2$-valued measure $\nu$, expressed as a
  $2\times 2$ matrix of scalar measures with respect to the standard
  orthonormal basis $\{e_1,e_2\}$ of $\mathbb C^2$ has the form
  \begin{equation}
    \label{eqn:nu-entries}
    \nu = \begin{pmatrix} \nu_{11} & \nu_{12} \\ \nu_{21} & \nu_{22}
    \end{pmatrix}.
  \end{equation}
  Since $\nu(\omega)$ is a positive matrix for every measurable set
  $\omega$, it follows that $\nu_{11}, \nu_{22}$ are positive
  measures.  Moreover for the off-diagonal entries we have
  $\nu_{21}=\nu_{12}^*$.  If $\omega$ is such that
  $\nu_{11}(\omega)=0$, then by positivity $\nu_{12}(\omega)=0$, and
  similarly if $\nu_{22}(\omega)=0$.  So it follows that $\nu_{12}$
  and $\nu_{21}$ are absolutely continuous with respect to both
  $\nu_{11}$ and $\nu_{22}$.  This argument also shows that $\nu_{12}$
  and $\nu_{21}$ are supported on the intersection of the supports for
  $\nu_{11}$ and $\nu_{22}$.

  Choosing $\gamma=e_1$, the hypotheses imply there exist
  $\alpha_1,\alpha_2\ge 0$ and points $\fz_1,\fz_2$ such that
  \begin{equation*}
    \nu_{11} = \sum_{j=1}^2 \alpha_j \delta_{\fz_j}.
  \end{equation*}
  Likewise there exist points $\fw_1,\fw_2$ and scalars
  $\beta_1,\beta_2\ge 0$ such that
  \begin{equation*}
    \nu_{22} = \sum_{j=1}^2 \beta_j \delta_{\fw_j}.
  \end{equation*}

  There are several cases to consider.  First suppose that the
  $\{\fz_1,\fz_2\}$ and $\{\fw_1,\fw_2\}$ have no points in common.
  Then $\nu_{12}=0=\nu_{21}$.  Also, for $\gamma=e_1+e_2$, by
  assumption
  \begin{equation*}
    \nu_\gamma = \nu_{11}+\nu_{22}
  \end{equation*}
  has support at two points, and so $\fz_1=\fz_2$ and $\fw_1=\fw_2$.
  It follows that the union of the supports of $\nu_{11}$ and
  $\nu_{22}$ has cardinality at most two and $\nu_{12}=0$, yielding
  the desired result.

  Next suppose that the sets $\{\fz_1,\fz_2\}$ and $\{\fw_1,\fw_2\}$
  have one point in common, say $\fz_1=\fw_1$.  In this case
  $\nu_{12}$ is supported at $\fz_1$ and there is a complex number $s$
  so that
  \begin{equation*}
    \nu_{12} = s \delta_{\fz_1}.
  \end{equation*}
  If $s=0$, choose $\gamma = e_1+e_2$, so that $\nu_\gamma =
  \nu_{11}+\nu_{22}$.  Otherwise set $\gamma= e_1 + s^* e_2$, in which
  case,
  \begin{equation*}
    \nu_{\gamma} = \nu_{11} + 2 |s|^2 \delta_{\fz_1} + |s|^2 \nu_{22}.
  \end{equation*}
  In either case, $\nu_\gamma$ has support at $\{\fz_1,\fz_2,\fw_2\}$
  and only two of these can be distinct.  

  The remaining case has the sets $\{\fz_1,\fz_2\}$ and
  $\{\fw_1,\fw_2\}$ equal, and the result is immediate.

  Positivity of $\nu$ implies
  positivity of $Q_1$ and $Q_2$.
\end{proof}

\section{The proof of Theorem \ref{thm:mainneg}}
 \label{sec:proof}
Fix  a finite set $\fs$ \index{$\fs$} containing
$0,\lambda_1,\lambda_2$ and consisting of at least six distinct
points. This choice of $\fs$ along with the prior choices of
  $\Phi$ and $F$ as in Equations \eqref{eq:Phi} and \eqref{eq:bF}
  remain in effect for the rest of the paper. 
 Accordingly, let $\Sigma_F=\Sigma_{F,\fs}$. 

 We next prove the following diagonalization result.

\begin{theorem}
  \label{thm:cone-must-diagonalize}
  If $\Sigma_F$ lies in the cone $C_{2,\fs}$, that is there exists an
  $M_2(\mathbb C)$ valued  $\mu$ such that 
  \begin{equation}
   \label{eq:Fmu}
    I-F(x)F(y)^* = \int_{\Psi} (1-\psi(x)\psi(y)^*)\,d\mu_{x,y}(\psi)
    \qquad x,y\in \fs,
  \end{equation}
  then there exists
  rank one orthogonal projections $Q_1, Q_2$
  summing to $I$, such that, for $x,y\in\fs$,
  \begin{equation}
    \label{eq:Fmubetter}
    I-F(x)F(y)^* = (1-x^2y^{*2}\varphi_1(x)\varphi_1(y)^*)Q_1 +
    (1-x^2y^{*2}\varphi_2(x)\varphi_2(y)^*)Q_2.
  \end{equation}
\end{theorem}

The proof proceeds by a sequence of lemmas which increasingly restrict
the measures $\mu_{x,y}$ in \eqref{eq:Fmu}.

Assume that $\Sigma_F\in C_{2,\fs}$.  Multiplying (\ref{eq:Fmu}) by
the Szeg\H{o} kernel $s(x,y)=(1-xy^*)^{-1}$ obtains
\begin{equation}
  \label{eqn:cone-times-szego}
  \left(\frac{I-F(x)F(y)^*}{1-xy^*}\right)_{x,y\in\fs} 
  = {\left(\int_{\Psi} \left(\frac{1-\psi(x)\psi(y)^*}{1-xy^*}\right)
      \,d\mu_{x,y}(\psi) \right)}_{x,y\in \fs}.
\end{equation}
Next, since $F$ has the form $x^2\Phi(x)$,
\begin{equation*}
 \begin{split}
  \frac{I-F(x)F(y)^*}{1-xy^*}
    =\, & \frac{I_2-x^2y^{*2}I_2  +x^2y^{*2}I_2 - x^2y^{*2}
      \Phi(x)\Phi(y)^*}{1-xy^*} \\
    =\, & (1+xy^*)I_2 +x^2y^{*2} \left(\frac{I -
        \Phi(x)\Phi(y)^*}{1-xy^*}\right).
  \end{split}
\end{equation*}
Similarly, for the test functions $\psi_\lambda(x)
=x^2\varphi_\lambda(x)$ at points $\lambda\in \mathbb D$,
\begin{equation}
 \label{eq:test-fn-szego}
  \frac{1-\psi_\lambda(x)\psi_\lambda(y)^*}{1-xy^*} =
  (1+xy^*) +x^2y^{*2}
  \left(\frac{1 - \varphi_\lambda(x) \varphi_\lambda(y)^*}{1 -
      xy^*}\right).
\end{equation}
(Here we take $\varphi_\infty = 1$.)  Letting
\begin{equation*}
 k_\lambda(x) = \frac{\sqrt{1-|\lambda|^2}}{1-\lambda^* x}
\end{equation*}
denote the normalized Szeg\H{o} kernel at $\lambda$ and using the
identity
\begin{equation}
 \label{eq:klambda-philambda}
   \frac{1-\varphi_\lambda(x)\varphi_\lambda(y)^*}{1-xy^*} =
   k_\lambda(x)k_\lambda(y)^*,
\end{equation}
for $\lambda \neq \infty$, equation \eqref{eq:test-fn-szego} gives,
\begin{equation*}
  \frac{1-\psi_\lambda(x)\psi_\lambda(y)^*}{1-xy^*}
  = (1+xy^*) + x^2y^{*2} k_\lambda(x)k_\lambda(y)^*,
\end{equation*}
while for $\lambda = \infty$ (correspondingly, $\psi_\infty(z) =
z^2$ and $k_\infty(x) = 0$),
\begin{equation*}
  \frac{1-\psi_\infty(x)\psi_\infty(y)^*}{1-xy^*}
  =1+xy^*.
\end{equation*}
Putting these computations together, we rewrite
(\ref{eqn:cone-times-szego}) as
\begin{equation}
  \label{eqn:cone-times-szego-alt}
  \begin{split}
    \frac{I-F(x)F(y)^*}{1-xy^*} &= (1+xy^*)I_2 + x^2y^{*2}
    \left(\frac{I-\Phi(x)\Phi(y)^*}{1 -xy^*}\right) \\
    & = (1+xy^*)\int_{\Psi}\, d\mu_{x,y}(\psi) +
    x^2y^{*2}\int_{\Psi_0} k_\lambda(x)k_\lambda(y)^* \,
    d\mu_{x,y}(\psi).
  \end{split}
\end{equation}
Note that the first integral is over $\Psi$ while the second is just
over $\Psi_0=\Psi\backslash\{z^2\}$ since $k_\infty(x)=0.$

Combining Lemma \ref{lem:I-PP*} with Equation
\eqref{eqn:cone-times-szego-alt} gives
\begin{equation}\label{eqn:cone-times-szego-ab}
 \begin{split}
  (1+xy^*)I + & x^2 y^{*2} \left(a(x)a(y)^* + b(x)b(y)^*\right) \\ &  = 
   \int_{\Psi}  (1+xy^*)\,d\mu_{x,y}(\psi) + \int_{\Psi_0}x^2y^{*2}\,
   k_\lambda(x) k_\lambda(y)^* \, d\mu_{x,y}(\psi).
 \end{split}
\end{equation}

The next step will be to remove the $x,y$ dependence in $\mu$.
Introducing some notation, let
\begin{equation*}
 \begin{split}
  \tilde{A}(x,y) = & \int_{\Psi} d\mu_{x,y}(\psi); \\
  R(x,y) = & x^2 y^{*2} \left(a(x)a(y)^* + b(x)b(y)^*\right);
  \quad\text{ and } \\
  \tilde{R}(x,y) = & (xy^*)^2 \int_{\Psi_0} k_\lambda(x)
  k_\lambda(y)^* \, d\mu_{x,y}(\psi).
 \end{split}
\end{equation*}
Thus, $\tilde{A},$ $R,$ and $\tilde{R}$ are all positive kernels on
$\fs$.  With this notation and some rearranging
 of Equation \eqref{eqn:cone-times-szego-ab}, for $x,y\in\fs$,
\begin{equation}
 \label{eq:allwegot}
  (1+xy^*)(\tilde{A}(x,y)-I) = R(x,y) -\tilde{R}(x,y).
\end{equation}

Let
\begin{equation}
  \label{eq:mathbbK}
  \mathbb K = \{x^2 k_{\lambda_1}(x)v_1, x^2k_{\lambda_2}(x)v_2\},
\end{equation}
the set of vectors spanning the kernel of $I-M_\Phi M_\Phi^*$
appearing in Lemma \ref{lem:I-PP*}.

\begin{lemma}
 \label{lem:allwegot}
  With the above notations,  the assumption that $\Sigma_F\in
  C_{2,\fs}$ and for $x,y\in\fs$,
  \begin{enumerate}[(i)]
  \item The $M_2(\mathbb C)$ valued kernel $(\tilde{A}-I)(x,y) =
    \tilde{A}(x,y)-I$ is positive semidefinite;
  \item
    The $M_2(\mathbb C)$ valued kernel $R(x,y) -\tilde{R}(x,y)$ is
    positive semidefinite with rank at most two;
  \item The range of $\tilde{R}$ lies in the range of $R$, which is in
    the span of $\mathbb K$; and
  \item Either
    \begin{enumerate}[(a)]
    \item The kernel $\tilde{A}-I$ has rank at most one; i.e., there
      is a function $r:\fs\to \mathbb C^2$ such that
      \begin{equation}
        \label{eq:rankAt}
        \tilde{A}(x,y) = I + r(x)r(y)^*,  \quad\mbox{ or;}
      \end{equation}
    \item there exist functions $r,s:\fs\to\mathbb C^2$ such
      \begin{equation*}
        \tilde{A}(x,y) = I + r(x)r(y)^* +s(x)s(y)^*,
      \end{equation*}
      and a point $\fz\in \fs\setminus \{0\}$ such that
      $r(\fz)=0=s(\fz)$.
    \end{enumerate}
  \end{enumerate}
\end{lemma}

\begin{proof}
  Since $\psi(0) = 0$ for all $\psi\in \Psi$, it follows from
  \eqref{eq:Fmu} for all $y\in\fs$,
  \begin{equation*}
    I = I-F(0)F(y)^* = \int_\Psi (1-\psi(0)\psi(y)^*) d\mu_{0,y}(\psi)
    = \int d\mu_{0,y}(\lambda)  = \tilde A(0,y)
  \end{equation*}
  and (\textit{i}) follows.

  That $R-\tilde{R}$ is positive semidefinite follows from item
  (\textit{i}) and Equation \eqref{eq:allwegot}.  Since $R$ is rank
  two it must be the case that the rank of $R-\tilde{R}$ is rank at
  most two, completing the proof of item (\textit{ii}).

  By item (\textit{ii}) and Douglas' lemma, the range of $\tilde{R}$
  is contained in the range of $R$.  By Lemma \ref{lem:I-PP*}, the
  range of $R$ is spanned by the set $\mathbb K$ and (\textit{iii})
  follows.

  To prove item (\textit{iv}), first note that in any case Equation
  \eqref{eq:allwegot} and item (\textit{ii}) imply $\tilde{A}-I$ has
  at most rank two; i.e., there exists $r,s:\fs\to\mathbb C^2$ such
  that
  \begin{equation*}
    \tilde{A}-I = r(x)r(y)^* +s(x)s(y)^*.
  \end{equation*}
  From Equation \eqref{eq:allwegot}, each of $r,xr,s,xs$ lie in the
  range of $R$, which equals the span of $\mathbb K$.  If $r$ is
  nonzero at two points in $\fs$, then $r$ and $xr$ are linearly
  independent and hence span the range of $R$.  In this case, as both
  $s$ and $xs$ are in the range of $R$ there exists $\alpha_j$ and
  $\beta_j$ (for $j=1,2$) such that
  \begin{equation*}
    \begin{split}
      s = & \alpha_1 r  + \alpha_2 xr \\
      xs = & \beta_1 r + \beta_2 xr.
    \end{split}
  \end{equation*} 
  It follows that
  \begin{equation}
    \label{eq:rands}
    0 = xs-xs = (\beta_1 + (\beta_2-\alpha_1)x + \alpha_2 x^2) r(x).
  \end{equation}
  If $\alpha_2=0$, then $s$ is a multiple of $r$ and case
  (\textit{iv})(\textit{a}) holds.  Otherwise, in view of
  \eqref{eq:rands}, $r$ is zero with the exception of at most two
  points.  Thus $r$ is zero at two points, one of which, say $\fz$,
  must be different from $0$.  Since $s$ must be zero when $r$ is,
  $s(\fz)=0$ too and (\textit{iv})(\textit{b}) holds.

  The remaining possibility is that both $r$ and $s$ are non-zero at
  at most one point each, and these points may be distinct.  In this
  situation $r$ and $s$ have at least two common zeros, one of which
  must be different from $0$ and again (\textit{iv})(\textit{b})
  holds.
\end{proof}

\begin{lemma}
  \label{lem:At-is-I}
  Under the assumption that $\Sigma_F\in C_{2,\fs}$, the $2\times 2$
  matrix-valued kernel $\tilde{A}$ is constantly equal to $I$; i.e.,
  $\At(x,y)= I_2$ for all $x,y\in\fs.$
\end{lemma}

\begin{proof}
  In the case that (\textit{iv})(\textit{a}) holds in Lemma
  \ref{lem:allwegot}, it (more than) suffices to prove that the $r$ in
  Equation \eqref{eq:rankAt} is $0$.  To this end, let $\rR$ denote
  the range of $R$ which, by Lemma \ref{lem:allwegot}, is spanned by
  the set $\mathbb K$ appearing in Equation \eqref{eq:mathbbK}.  From
  Equations \eqref{eq:allwegot} and \eqref{eq:rankAt},
  \begin{equation*}
    \tilde{R} +   (1+xy^*)r(x)r(y)^* = R.
  \end{equation*}
  Thus, $\rR$ contains both $r$ and $xr$; that is, both $r$ and $xr$
  are in the span of $\mathbb K$.  Consequently, there exists
  $\alpha_j$ and $\beta_j$ ($j=1,2$) such that
  \begin{equation*}
    \begin{split}
      r = & x^2 \sum_{j=1}^2 \alpha_j k_{\lambda_j}(x) v_j \\
      xr = & x^2 \sum_{j=1}^2 \beta_j k_{\lambda_j}(x)v_j.
    \end{split}
  \end{equation*}
  Hence,
  \begin{equation}
    \label{eq:premystery}
    0 = xr-xr = x^2 \sum_{j=1}^2 (\beta_j -x \alpha_j)k_{\lambda_j}(x)
    v_j.
  \end{equation}
  Since the set $\{v_1,v_2\}$ is a basis for $\mathbb C^2$ (see Lemma
  \ref{lem:I-PP*}), it has a dual basis $\{w_1,w_2\}$.  Taking the
  inner product with $w_\ell$ in Equation \eqref{eq:premystery} gives,
  \begin{equation*}
    0 = x^2 (\beta_\ell-x\alpha_\ell)k_{\lambda_\ell}(x)
  \end{equation*}
  for $x\in \fs$.  Choosing $x=\lambda_\ell$ (which is not zero)
  implies $\beta_\ell-\lambda_\ell \alpha_\ell=0$.  But then choosing
  any $x\in \fs$ different from both $0$ and $\lambda_j$ (and using
  $k_{\lambda_j}(x)\ne 0$) implies $\beta_\ell-x\alpha_\ell=0$.  Hence
  $\alpha_\ell=0=\beta_\ell$ and consequently $r(x)=0$ for all $x$.

  Now suppose (\textit{iv})(\textit{b}) in Lemma \ref{lem:allwegot}
  holds.  In particular, there exists a point $\fz$ in $\fs \setminus
  \{0\}$ such that $r(\fz)=0=s(\fz)$.  By the same reasoning as in the
  first part of this proof, there exist $\alpha_j$ and $\beta_j$ such
  that
\begin{equation*}
 \begin{split}
     r = & x^2 \sum_{j=1}^2 \alpha_j k_{\lambda_j}(x) v_j \\
     s = & x^2 \sum_{j=1}^2 \beta_j k_{\lambda_j}(x) v_j.
 \end{split}
\end{equation*}
  Taking the inner product with $w_\ell$ and evaluating at $\fz$ yields
\begin{equation*}
   0= \alpha_\ell k_{\lambda_\ell}(\fz).
\end{equation*}
Thus $\alpha_\ell=0$.  Likewise, $\beta_\ell=0$.  Thus $r=0=s$ and the
proof is complete.
\end{proof}

\begin{remark}
  \label{rem:nomoremystery}
  Observe that if it were the case that $v_1=v_2$ in Equation
  \eqref{eq:premystery}, then it would not be possible to conclude
  that the $\alpha_j$ and $\beta_j$ are $0$.  Indeed, in such a
  situation, choosing $\beta_j=(-1)^j$ and
  $\alpha_j=(-1)^j\lambda_j^*$ gives a non-trivial solution.  However,
  the case $v_1=v_2$ corresponds to a $\Phi$ having the form
  \begin{equation*}
    \Phi = \begin{pmatrix} 1 & 0 \\ 0 & \varphi_{\lambda_1}
      \varphi_{\lambda_2} \end{pmatrix},
  \end{equation*}
  which is explicitly ruled out by our choice of $\Phi$ and
  Lemma~\ref{lem:badPhi}.
\end{remark}

\begin{lemma}
  \label{lem:prediag-cone}
  There exists a $2\times 2$ matrix valued positive measure $\mu$ on
  $\Psi$ such that $\mu(\Psi)=I_2$ and
  \begin{equation}\label{eqn:prediag-cone}
    K^\Phi(x,y):= \frac{1-\Phi(x)\Phi(y)^*}{1-xy^*} =
    \int_{\Psi_0} k_\lambda(x)k_\lambda(y)^* \,d\mu(\psi) 
  \end{equation}
  for all $x,y\in\fs\setminus\{0\}$. 
\end{lemma}

\begin{proof}
  By Lemma~\ref{lem:At-is-I}, $\tilde{A}(x,y)=I$ for all $x,y\in\fs$.
  An examination of the definition of $\tilde{A}$ and application of 
  Lemma~\ref{lem:posI}  implies there is a positive
  measure $\mu$ such that $\mu_{x,y}=\mu$ 
  for all $(x,y).$  Substituting this representation 
  for $\mu_{x,y}$ into and some canceling and  rearranging
  of  \eqref{eqn:cone-times-szego-alt}  gives,
  \begin{equation*}
    (xy^*)^2 \left(\frac{I-\Phi(x)\Phi(y)^*}{1 -xy^*}\right) 
    =  x^2y^{*2}\int_{\Psi_0} k_\lambda(x)k_\lambda(y)^* \,
    d\mu(\psi).
  \end{equation*}
  Dividing by $(xy^*)^2$ (and of course excluding either $x=0$ or
  $y=0$) gives the result.
 \end{proof}

 Now that $\mu$ has no $x,y$ dependence, the next step is to restrict
 its support.  For this we employ Lemma \ref{lem:I-PP*}.
Recall that $\mu$ is a positive
$2\times 2$ matrix-valued measure on $\Psi$.  Let $\delta_{\infty}$
denote point mass at the point $\psi_\infty = z^2$.

\begin{lemma}
  \label{lem:pre-diag1} 
  Under the assumption that $\Sigma_F\in C_{2,\fs}$, and with notation
  as above, there are two points $\fz_1,\fz_2$ in $\fs$ such that the
  measure $\mu$ has the form $\mu = \delta_{\fz_1}Q_1 +
  \delta_{\fz_2}Q_2 + \delta_{\infty} P$, where $Q_1, Q_2, P$ are
  $2\times 2$ matrices satisfying $0\leq Q_1, Q_2, P\leq 1$ and
  $Q_1+Q_2+P=I$, and $\delta_{\fz_1}, \delta_{\fz_2}$ are scalar unit
  point measures on $\Psi$ supported at $\psi_{\fz_1}, \psi_{\fz_2}$,
  respectively.
\end{lemma}

\begin{proof}
  We first show that the restriction of $\mu$ to $\mathbb D$ has
  support at no more than two points.  Accordingly, let $\nu$ denote
  the restriction of $\mu$ to $\mathbb D$.

  From Lemma \ref{lem:I-PP*}, for $x,y\in \fs\setminus \{0\}$,
  \begin{equation*}
    \frac{I_2-\Phi(x)\Phi(y)^*}{1-xy^*}=a(x)a(y)^*+b(x)b(y)^*
  \end{equation*}
  where $a,b$ are $\mathbb C^2$ valued functions on $\bfs$.  Fix a
  vector $\gamma$ and define a scalar measure $\nu_\gamma$ on $\Psi$
  by $\nu_\gamma(\omega) = \gamma^* \nu(\omega)\gamma$.  Note that
  \begin{equation*}
    \begin{split}
      \gamma^* \left(a(x)a(y)^* + b(x)b(y)^*\right) \gamma = &
      \gamma^* \left(\int_{\Psi} k_\lambda(x)k_\lambda(y)^*  d\mu
        (\psi) \right)\gamma \\
      =& \int_{\Psi_0} k_\lambda(x)k_\lambda(y)  d\nu_\gamma (\psi)
    \end{split}
  \end{equation*}
  is a kernel of  rank (at most) two. 

  Choosing a three-point subset $\mathfrak G\subset\fs\setminus \{0\}$
  and a nonzero scalar-valued function $c:\mathfrak G\to \mathbb C$
  such that
  \begin{equation}
    \label{eq:3}
    \sum_{x,y\in \mathfrak{G}} c(x)\gamma^* \left(a(x)a(y)^* +
      b(x)b(y)^*\right) c(y)^* =0
  \end{equation}
 gives
  \begin{equation}
    \label{eq:2}
    0 = \int_{\Psi_0} {\left|\sum_{x\in\mathfrak{G}}
      k_\lambda(x)c(x)\right|}^2\, d\nu_\gamma (\psi),
  \end{equation}
  which means that the function $f=\sum_{x\in\mathfrak{G}}
  k_\lambda(x)c(x)$ vanishes for $\nu_\gamma$-a.e.~on $\Psi_0$.  The
  function $f$ is a linear combination of at most three Szeg\H{o}
  kernels, and hence can vanish at at most two points in $\mathbb D$.
  It follows that $\nu_\gamma$ is supported at at most two points in
  $\mathbb D$.  An application of Lemma \ref{lem:finiteistwo} now
  implies that there exist points $\fz_1,\fz_2$ and positive
  semidefinite matrices $Q_1,Q_2$ such that
  \begin{equation*}
    \nu = \sum_{j=1}^2 \delta_{\fz_j} Q_j.
  \end{equation*}
  Letting $P=\mu(\{\infty\})$, it follows that $\mu$ has the promised
  form,
  \begin{equation*}
    \mu = \delta_{\fz_1}Q_1 + \delta_{\fz_2}Q_2 +\delta_{\infty} P.
  \end{equation*}
  Finally, because $\mu$ has total mass the identity, 
  \begin{equation*}
    I=\mu(\Psi) = Q_1+Q_2+P.\qedhere
  \end{equation*}
\end{proof}

To eliminate $P$ and show that the $Q_i$ are orthogonal, rank one
projections, return to Equation (\ref{eqn:prediag-cone}) and rearrange
it once again: recalling the identity of Equation
\eqref{eq:klambda-philambda} and multiplying through by $1-xy^*$ and
using Lemma \ref{lem:pre-diag1}, we have by the description of $\mu$
from the previous lemma, for $x,y\in\fs\setminus\{0\}$,
\begin{equation*}
  1-\Phi(x)\Phi(y)^* =
  (1-\varphi_{\fz_1}(x)\varphi_{\fz_1}(y)^*)Q_1 +
  (1-\varphi_{\fz_2}(x)\varphi_{\fz_2}(y)^*)Q_2,
\end{equation*}
where $\psi_{\fz_1}, \psi_{\fz_2}$ are the support points of the
measure $\mu$. 
Using the fact that $Q_1+Q_2+P=I$, we obtain, for $x,y\in\fs\setminus
\{0\}$,
\begin{equation}
  \label{eqn:prediag-cone-reallysimple}
  \Phi(x)\Phi(y)^* = \varphi_{\fz_1}(x)\varphi_{\fz_1}(y)^*Q_1 +
  \varphi_{\fz_2}(x)\varphi_{\fz_2}(y)^*Q_2 +P.
\end{equation}

\begin{lemma}\label{lem:prediag-final}
  Let $\Phi$ be as above.  In the
  representation~(\ref{eqn:prediag-cone-reallysimple}),
  \begin{enumerate}[(i)]
  \item $\{\fz_1,\fz_2\}=\{\lambda_1,\lambda_2\}$;
  \item  $P=0$; and 
  \item $Q_1,Q_2$ are rank one projections summing to $I$ $($and hence
    mutually orthogonal$)$.
 \end{enumerate}
\end{lemma}

\begin{proof}
  Since $\det \Phi(\lambda_1) =0$, the identity 
  (\ref{eqn:prediag-cone-reallysimple}) implies
  \begin{equation*}
    \Phi(\lambda_1)\Phi(\lambda_1)^* =
    |\varphi_{\fz_1}(\lambda_1)|^2Q_1
    +|\varphi_{\fz_2}(\lambda_1)|^2 Q_2+P,
  \end{equation*}
  that  both sides have rank at most one.  It follows that at least one
  of $\varphi_{\fz_1}, \varphi_{\fz_2}$ (and hence exactly one, since
  the $\lambda_j$ are distinct) must have a zero at $\lambda_1$
  (otherwise the three positive matrices $Q_1,Q_2,P$ would all be
  scalar multiples of the same rank one matrix, which violates
  $Q_1+Q_2+P=I$).  Similarly for $\lambda_2$, so (\textit{i}) is
  proved.  Further, without loss of generality, it can be assumed that
  $\fz_j=\lambda_j$ for $j=1,2$.

  It follows from evaluating at the $\lambda_j$ that each of $Q_1,
  Q_2,P$ has rank at most one.  In particular we have for $j=1,2$,
  \begin{equation}
    \label{eq:1}
    \Phi(\lambda_j)\Phi(\lambda_j)^* = |\varphi_k(\lambda_j)|^2 Q_k
    +P,
  \end{equation}
  where $k\in\{1,2\}$ and $k\neq j$.  This means that $\text{ran}
  P\subset\text{ran} Q_1\cap\text{ran} Q_2$.  On the other hand, if
  $\text{ran} Q_1\cap\text{ran} Q_2 \neq \{0\}$, we have $\text{ran}
  Q_1\subseteq\text{ran} Q_2$ or vice versa, which again contradicts
  $Q_1+Q_2+P=1$.  Thus $\text{ran} Q_1\vee\text{ran} Q_2 = \mathbb
  C^2$, and so $P=0$, which is (\textit{ii}).  Since $Q_2 = 1-Q_1$, if
  $f\in\text{ker} Q_1$, then $Q_2f = f$.  However, $Q_2$ is a rank one
  contraction, so it must be a projection, and then the same follows
  for $Q_1$.  Thus we have~(\textit{iii}).
\end{proof}

\begin{proof}[Proof of Theorem~\ref{thm:cone-must-diagonalize}] 
  Since $F(x)=x^2\Phi(x)$, Theorem~\ref{thm:cone-must-diagonalize} is
  now immediate from Lemma~\ref{lem:prediag-final}.
\end{proof}

\subsection{The proof of Theorem \ref{thm:mainneg}}
\label{subsec:proof}

The proof of Theorem \ref{thm:mainneg} concludes in this subsection.
Recall that we are assuming that $F(z) = z^2\Phi(z)$, where $\Phi$ is
as in \eqref{eq:Phi}.

Suppose that $\Sigma_F \in C_{2,\fs}$. 
From Equation \eqref{eqn:prediag-cone-reallysimple} and Lemma
\ref{lem:prediag-final},
\begin{equation}
  \label{eq:diag}
  \Phi(x)\Phi(y)^* = \sum_{j=1}^2 \varphi_j(x) \varphi_j(y)^* Q_j,
\end{equation}
valid for $x,y\in\fs\setminus \{0\}$.  Since the $Q_j$ are rank one
projections which sum to $I$, there exists an orthonormal basis
$\{\gamma_1,\gamma_2\}$ such that
\begin{equation*}
  Q_j = \gamma_j \gamma_j^*.
\end{equation*}
Let $U$ be the unitary matrix with columns $\gamma_j$, and let
\begin{equation*}
 G(z) = U\begin{pmatrix} \varphi_1(z)& 0 \\ 0 & \varphi_2(z)
 \end{pmatrix}.
\end{equation*}
Observe $\Phi(x)\Phi(y)^* = G(x)G(y)^*$ for $x,y\in\fs\setminus
\{0\}.$

Fix $\zeta\in\fs\setminus \{0,\lambda_1,\lambda_2\}$.  Then
$\Phi(\zeta)$ is invertible and further $\Phi(\zeta)\Phi(\zeta)^* =
G(\zeta)G(\zeta)^*$.  Hence by Douglas' Lemma, there is a unitary $W$
such that $\Phi(\zeta)=G(\zeta)W^*$.  Consequently,
\begin{equation*}
 0= \Phi(\zeta)\Phi(y)^* - G(\zeta)G(y)^* = G(\zeta){\left(\Phi(y)W
     -G(y)\right)}^*,
\end{equation*}
and therefore $\Phi(y)W=G(y)$, for $y\in\fs\setminus\{0\}$.  Returning
to the definition of $G$, we arrive at the conclusion that, for
$x\in \fs\setminus \{0\}$,
\begin{equation}
 \label{eq:contra}
 \Phi(x) = U\begin{pmatrix} \varphi_1(x) & 0 \\ 0 & \varphi_2(x)
 \end{pmatrix} W^*.
\end{equation}
Now $\Phi$ and $G$ are both rational matrix inner functions of degree at
most two.  Since $\fs\setminus \{0\}$ contains at least five points 
it is a set of uniqueness for rational functions of
degree at most two, and hence \eqref{eq:contra} must hold on all of
$\mathbb D$.  Returning to $\Phi$, it now follows that, on all of
$\mathbb D$,
\begin{equation*}
 \Phi = U\begin{pmatrix} \varphi_1 & 0 \\ 0 & \varphi_2\end{pmatrix}
 W^*.
\end{equation*}
By Lemma \ref{lem:badPhi},
\begin{equation*}
 \Phi = \begin{pmatrix} s\varphi_1 & 0 \\ 0 & t\varphi_2 \end{pmatrix}
\end{equation*}
for unimodular constants $s$and $t$, contrary to our choice of $\Phi$
in \eqref{eq:Phi}.  We conclude that $\Sigma_F \notin C_{2,\fs}$, and
so by Proposition~\ref{prop:counterifsep}, there exists a contractive
representation of $\cA$ which is contractive, but not completely
contractive.

\section{Rational Dilation for the annulus and the variety $z^2=w^2$}
\label{sec:ratforA}

This section provides a proof of rational dilation for the annulus
along the lines of \cite{M}, but with a major simplification suggested
by Agler \cite{Agprivate} (see also \cite{AHR}). Direct appeal to the
systematic study of the extreme rays of functions of positive real
part on a multiply connected domain found in \cite{DP} \cite{BH} and
\cite{BH2} (see also \cite{H} \cite{DMrat}) also significantly
streamline the argument.  This proof for the annulus, with minor
modifications indicated in Subsection \ref{sec:z2=w2}, also
establishes rational dilation for the distinguished variety defined by
$z^2=w^2$.

\subsection{A Naimark Dilation Theorem}

The following version of the Naimark Dilation Theorem will be used to
reduce the extreme rays of functions of positive real part on an
annulus to a much smaller collection.

\begin{theorem}
  \label{thm:naimark}
  Fix positive integers $m,n$ and suppose that $A_1,\dots,A_m;
  B_1,\dots, B_m$ are rank one positive semidefinite $n\times n$
  matrices. If
  \begin{equation*}
    \sum A_j = I=\sum B_\ell,
  \end{equation*}
  then there exists an isometry $V:\mathbb C^n\to \mathbb C^m$ and
  $m\times m$ matrices $P_1,\dots,P_m; Q_1\dots, Q_m$ such that
  \begin{enumerate}[(i)]
  \item Each of $P_1,\dots, P_m;Q_1,\dots Q_m$ are rank one projections;
  \item 
    \begin{equation*} 
      \sum P_j = I =\sum Q_\ell;
    \end{equation*}
  \item and 
    \begin{equation*}
      A_j = V^* P_jV, \ \  B_\ell= V^* Q_\ell V.
    \end{equation*}
  \end{enumerate}
\end{theorem}

\begin{proof}
  Since the $A_j$ are rank one and positive semidefinite, there exists
  $a_j\in\mathbb C^n$ such that
  \begin{equation*}
    A_j = a_j a_j^*.
  \end{equation*}
  Let $V$ denote the matrix whose $j$-th row is $a_j^*$ (the $1\times
  n$) matrix. It follows that $V$ is an $m\times n$ matrix and
  moreover,
  \begin{equation*}
    V^* V = \sum a_j a_j^* = I.
  \end{equation*}
  Thus $V$ is an isometry.  Let $P_j = e_j e_j^*$, where $\{e_1,\dots,
  e_n\}$ is the standard orthonormal basis for $\mathbb C^n$ and note
  that
  \begin{equation*}
    V^* P_j V = a_j a_j^* = A_j.
  \end{equation*}

  The analogous construction with $B_j=b_jb_j^*$ produces an isometry
  $W:\mathbb R^n\to\mathbb R^m$ such that
  \begin{equation*}
    W^* P_\ell W = B_\ell.
  \end{equation*}
  
  Since $V$ and $W$ are isometries, the mapping $U:\mbox{range}(V) \to
  \mbox{range}(W)$ defined by $UVx= Wx$ is a unitary mapping. Since
  the codimensions of the range of $V$ and the range of $W$ are the
  same, $U$ can be extended to a unitary mapping on $\mathbb C^m$.
  Let $Q_\ell = U^*P_\ell U.$ Then each $Q_\ell$ is a rank one
  projection, the $Q_j$ sum to the identity and
  \begin{equation*}
    V^* Q_\ell V = V^* U^* P_\ell U V = W^* P_\ell W = B_\ell.
  \end{equation*}
\end{proof}

\subsection{Extremal Functions of Positive Real Part}

As a special case of the results in \cite{BH,BH2}, the matrix-valued
functions of positive real part on an annulus are characterized.

Fix $0<q<1$ and let $\mathbb A$ denote the annulus,
\begin{equation*}
  \mathbb A = \{z\in\mathbb C: q<|z|<1\},
\end{equation*}
with its boundary components
\begin{equation*}
  \partial_0 =\{|z|=1\}, \ \ \partial_1 =\{|z|=q\}.
\end{equation*}
Let $M_n$ denote the $n\times n$ matrices. An analytic function
$F:\mathbb A\to M_n$ whose real part,
\begin{equation*}
  \mathrm{Re}\,F(z) = \frac{F(z)+F(z)^*}{2}
\end{equation*}
takes positive definite values in $\mathbb A$ has an $n\times n$
matrix-valued measure $\mu_F$ on $\partial= \partial_0\cup \partial_1$
for its boundary values.  On the other hand, a positive semidefinite
$n\times n$ matrix-valued measure $\mu$ on $\partial$ is the boundary
values of a matrix-valued harmonic function $H$ on $\mathbb A$.
Moreover, $H$ is the real part of analytic function if and only if
\begin{equation*}
  \mu(\partial_0)=\mu(\partial_1).
\end{equation*}
By compressing to the range of $\mu(\partial_0)$, it can be assumed
that $\mu(\partial_0)$ has full rank.

Let $\Gamma_n$ denote the set of positive semidefinite $n\times n$
matrix-valued measures $\mu$ on $\partial$ such that $\mu(\partial_0)=
I= \mu(\partial_1)$.  The results of \cite{BH} and \cite{BH2} imply
that the extreme points of the set $\Gamma_n$ have the form,
\begin{equation}
  \label{eq:extrememu}
  \mu = \sum_{j=1}^m A_j \delta_{\alpha_j} + \sum_{\ell=1}^m B_\ell
  \delta_{\beta_\ell},
\end{equation}
where the $A_j$ and $B_\ell$ are rank one and positive semidefinite
$n\times n$ matrices; $\alpha_j$ are points on $\partial_1$ and the
$\beta_\ell$ are points on $\partial_0;$ and
\begin{equation*}
  \sum A_j = I =\sum B_\ell.
\end{equation*}
Repetition is allowed among the points $\alpha$ and $\beta$ to allow
for attaching arbitrary positive semidefinite matrices to a point on
$\partial$ and the zero matrix is allowed for some the $A$ or $B$ so
that it may be assumed, without loss of generality, that there are the
same number of $A_j$ as there are of the $B_\ell.$ It should be noted
that not every measure of the form in Equation \eqref{eq:extrememu} is
an extreme point (a characterization is given in \cite{BH,BH2}).

For a $\mu\in \Gamma_n,$ let $F_\mu$ denote a corresponding analytic
function of positive real part. Thus, the real part of $F_\mu$ is the
harmonic functions whose boundary values are $\mu$.  Such an $F$ is
not unique, but any two differ by a matrix $C$ which is skew
self-adjoint, $C^*= - C$. Note that the real part of $F$ is zero
except at the $n$ points (counting multiplicity) in the support of
$\mu$ on each of the components of $\partial.$

An operator $T$ has $\mathbb A$ as a spectral set if $\sigma(T)\subset
\mathbb A$ and $\|f(T)\|\le 1$ for each analytic function $f:\mathbb
A\to \mathbb D$.  Here $\mathbb D$ is the unit disc, $\{z\in\mathbb C:
|z|<1\}$ and $\sigma(T)$ is the spectrum of $T.$ The following
proposition is a consequence of the results of \cite{BH,BH2}.

\begin{theorem}
  \label{thm:BH}
  Let $T$ be a operator on the Hilbert space $H$ with
  $\sigma(T)\subset\mathbb A$, and suppose $\mathbb A$ is a spectral
  set for $T$.  Then there exists a normal operator $N$ acting on a
  Hilbert space $K$ with $\sigma(N)\subset \partial$ and an isometry
  $V:H\to K$ such that $r(T) = V^* r(N) V$ for all rational functions
  $r$ with poles off the closure of $\mathbb A$ if and only if
  \begin{equation*}
    \frac{F_\mu(T) + F_\mu(T)^*}{2} \succeq 0
  \end{equation*}
  for each $n$ and each $\mu$ as in Equation \eqref{eq:extrememu}.
\end{theorem}

\begin{remark}\rm
  The first equivalent condition of the theorem says that $T$ has a
  rational dilation to a normal operator with spectrum in the boundary
  of $\mathbb A$.
\end{remark}

\subsection{Matrix Extreme Functions of Positive Real Part}

There is a particularly nice subset of the extreme points of
$\Gamma_n$ from which all the extreme points of $\Gamma_n$ can be
recovered in a canonical fashion.

Let $\mathcal E_n$ denote those elements $\nu$ of $\Gamma_n$ of the
form,
\begin{equation*}
  \nu = \sum_{j=1}^n  A_j \delta_{\alpha_j} + \sum_{\ell=1}^n B_\ell
  \delta_{\beta_\ell}.
\end{equation*}
In particular, the $A_j$ are rank one projections which sum to the
identity and likewise for the $B_\ell.$

\begin{lemma}
  \label{lem:perturb}
  Let 
  \begin{equation*}
    G= (F_\nu-I)(F_\nu +I)^{-1}.
  \end{equation*}
  For each $n\times n$ unitary matrix $U$ the function 
  \begin{equation*}
    \det(I-G(z)U)
  \end{equation*}
  has precisely $n$ zeros on each boundary component of $\partial$.
\end{lemma}

\begin{proof}
  Because $F$ has positive real part, $G$ is contractive-valued in
  $\mathbb A$. Further, as the real part of $F$ is $0$, except on a
  finite subset of $\partial$, the function $G$ is unitary-valued on a
  cofinite subset of $\partial$ and hence extends by reflection
  principle to a function analytic in the neighborhood of the closure
  of $\mathbb A$.

  On the other hand,
  \begin{equation*}
    F_\nu =  (I+G)(I-G)^{-1}.
  \end{equation*} 
  Thus the real part of $F_\nu$ is zero at $z$ unless $1$ is in the
  spectrum of $G(z)$. Hence, $\det(I-G(z))$ has exactly $n$ zeros
  (counting) multiplicity on each boundary component of $\mathbb A.$

  Let $\mathcal U$ denote the collection of $n\times n$ unitary
  matrices.  Let
  \begin{equation*}
    \mathcal U_k =\{U\in\mathcal U: \det(I-G(z)U) \mbox{ has $k$ zeros
      on } \partial_0\}.
  \end{equation*}
  Note that as the only zeros of $\det(I-G(z)U)$ can occur on the
  boundary, this number of zeros is stable with respect to small
  perturbations of $U$. Thus, $\mathcal U_k$ is open. But also
  $\mathcal U = \cup \mathcal U_k$ is compact, hence this union is
  finite.  Since $\mathcal U_n$ is not empty and $\mathcal U$ is
  connected, it follows that $\mathcal U=\mathcal U_n$.
\end{proof}

\subsection{Rational Dilation on $\mathbb A$}
 
\begin{theorem}[\cite{Agler,M}]
  \label{thm:annulus-dilation}
  If the operator $T$ has the annulus as a spectral set, then $T$ has
  a normal dilation to an operator with spectrum in the boundary of
  $\mathbb A$.
\end{theorem}

\begin{proof}
  It suffices to verify the second of the equivalent conditions in
  Theorem \ref{thm:BH}.  Accordingly, let such an $F_\mu$ be given.
  By Theorem \ref{thm:naimark}, there is an $m$, an isometry
  $V:\mathbb C^n\to\mathbb C^m$ and rank one projections $P_j$ and
  $Q_\ell$ as described in that theorem so that
  \begin{equation*}
    V^* P_jV = A_j, \ \  V^* Q_\ell V = B_\ell.
  \end{equation*}

  Consider the measure 
  \begin{equation*}
    \nu =\sum P_j \delta_{\alpha_j} + \sum Q_\ell \delta_{\beta_\ell}.
  \end{equation*}
  Because the $P_j$ and $Q_\ell$ each sum to the identity, it follows
  that there is an $m\times m$ matrix-valued analytic function $G$ of
  positive real part whose boundary values are the measure $\mu$.
  Further, since
  \begin{equation*}
    V^* \nu V = \mu,
  \end{equation*}
  if $\mathrm{Re}\,G(T)\succeq 0$, then also $\mathrm{Re}\,F(T)\succeq
  0$. Hence it suffices to prove $\mathrm{Re}\,G(T)\succeq 0$ under
  the assumption that $\mathbb A$ is a spectral set for $T.$

  Let
  \begin{equation*}
    G = (F_\nu -I)(F_\nu +I)^{-1}.
  \end{equation*}
  Thus $\Phi$ is a contractive analytic function in $\mathbb A$.  By
  Lemma \ref{lem:perturb}, with $U=G(1)^*$ and $\tilde{G}=GU$, the
  function $\det(I-G(z)U)$ has exactly $n$ zeros on each of
  $\partial_0$ and $\partial_1$. Let
  \begin{equation*}
    \tilde{F} = (I+\tilde{G})(I-\tilde{G})^{-1}.
  \end{equation*}
  Thus $\tilde{F}$ has positive real part and moreover its boundary
  values determine a measure $\tilde{\mu}$ with support at exactly
  (counting multiplicity) $n$ points on each boundary component. On
  the other hand, the choice of $U$ implies that
  \begin{equation*}
    \tilde{\mu} = \sum_{j=1}^n \tilde{A_j}\delta_{\gamma_j} +
    \tilde{B} \delta_1,
  \end{equation*}
  where the $\tilde{A_j}$ are rank one positive semidefinite $n\times
  n$ matrices and
  \begin{equation*}
    \tilde{B}=\sum \tilde{A_j} \succ 0.
  \end{equation*}
 
  Thus,
  \begin{equation*}
    \mathrm{Re}\,\tilde{F} = \sum_{j=1}^n
    \tilde{A_j}(\delta_{\gamma_j} +\delta_1)
  \end{equation*}
  on $\partial \mathbb A$.

  Consider now the scalar measures $\delta_{\gamma_j}+\delta_1$.  Each
  of these measure puts unit mass on each boundary component
  $\partial_0, \partial_1$, and hence for each $j$ there is a
  holomorphic function $\psi_j$ of positive real part in $\mathbb A$
  such that
  \begin{equation*}
    \mathrm{Re}\,\psi_j = \delta_{\gamma_j}+\delta_1
  \end{equation*}
  on $\partial \mathbb A$.  By the assumption that $\mathbb A$ is a
  spectral set for $T$, we have $\mathrm{Re}\,\psi_j(T)\geq 0$ for
  each $j$.

  Now let
  \begin{equation*}
    \Psi = \sum \psi_j \tilde{A_j}.
  \end{equation*}

  It follows that $\Psi$ and $\tilde{F}$ agree up to a skew symmetric
  matrix. Thus,
  \begin{equation*}
    \mathrm{Re}\,{\tilde{F}}(T) = \mathrm{Re}\,{\Psi}(T)=\sum
    (\mathrm{Re}\,\psi_j(T))\otimes \tilde{A_j} \succeq 0.
  \end{equation*}
  Hence,
  \begin{equation*}
    \|\tilde{G}(T)\|\le 1; 
  \end{equation*}
  and thus
  \begin{equation*}
    \|G(T)\|\le 1; 
  \end{equation*}
  and so finally,
  \begin{equation*}
    \mathrm{Re}\,F(T) \succeq 0,
  \end{equation*}
  and the proof is complete.  
\end{proof}

\subsection{The variety $z^2=w^2$} 
\label{sec:z2=w2}

In this section we consider the hypo-Dirichlet algebra $\mathcal
A(\mathcal V)$ of functions which are analytic on the variety
$\mathcal V=\{(z,w)\in \mathbb D^2:z^2=w^2\}$ in the bidisk and
continuous on its boundary $\partial \mathcal V$.  The boundary of
$\mathcal V$ consists of the two disjoint circles $z=w,$ and $z=-w$
with $|z|=|w|=1$.  The proof of rational dilation for the annulus
given above may be straightforwardly modified to prove the following
theorem:

\begin{theorem}\label{thm:V-dilation}
  Every contractive homomorphism $\pi:\mathcal A(\mathcal V)\to B(H)$
  is completely contractive.
\end{theorem}

The algebra $\mathcal A(\mathcal V)$ is in some sense a limiting case
of the annulus algebras. Indeed for a fixed real number $0<t<1$, the
variety in $\mathbb D^2$ defined by
\begin{equation}
  z^2 = \frac{w^2-t^2}{1-t^2w^2}  
\end{equation}
is an open Riemann surface which is topologically an annulus
\cite{Rudin}, and in fact by varying $t$ every annulus $\mathbb A_q$
is conformally equivalent to one of these \cite{Bell,Bell2}.  The
variety $\mathcal V$ is of course the limiting case $t\to 0$.

To get started we record some basic facts about $\mathcal V$ and
$\mathcal A(\mathcal V)$.  To fix some notation: $\mathcal V$ is the
union of the two sheets
\begin{equation}
  \mathcal V_+=\{(z,w):z=w\}, \quad \mathcal V_-=\{(z,w):z=-w\},
\end{equation}
which can each be identified with unit disk $\mathbb D$ via the
parametrization $\psi_+(t)=(t,t)$ and $\psi_-(t)=(t,-t)$
respectively.  The sheets $\mathcal V_\pm$ intersect only at the
origin, and the boundary of $\mathcal V\cap\mathbb D^2$ is the
disjoint union of the circles $\partial \mathcal V_+,$ and $\partial
\mathcal V_-$.  We equip each of these circles with normalized
Lebesgue measure (that is, the push-forward of Lebesgue measure under
the maps $\psi_{\pm}$).  We also recall that, by definition, a (scalar
or matrix valued) function $F$ is holomorphic on the variety $\mathcal
V$ if and only if for each point $(z,w)\in\mathcal V$, there is a
neighborhood $\Omega$ of this point in $\mathbb C^2$ such that $F$
extends to be holomorphic in $\Omega$.  $\mathcal A(\mathcal V)$ is
then the algebra of functions holomorphic on $\mathcal V$ and
continuous on $\mathcal V\cup \partial \mathcal V$, equipped with the
supremum norm, which we denote $\|f\|_{\mathcal V}$.

Given any function $F$ on $\mathcal V$, we let $F_\pm$ denote its
restrictions to the disks $\mathcal V_\pm$.  In particular, if $F$ is
holomorphic on $\mathcal V,$ then $H_\pm(t):=F_\pm(\psi_\pm(t))$ are
holomorphic functions on the disk, and $H_+(0)=H_-(0)$.  The converse
is also true:

\begin{lemma}
  \label{lem:V-extension}
  Given any pair of holomorphic functions $H_\pm:\mathbb D\to
  M_n(\mathbb C)$ with $H_+(0)=H_-(0)$, there exists a holomorphic
  function $F:\mathcal V\to M_n(\mathbb C)$ such that $F_\pm\circ
  \psi_\pm =H_\pm$.
\end{lemma}

\begin{proof}
  It suffices to assume $H_\pm(0)=0$, in which case the function
  \begin{equation}
    F(z,w) = (1-(z-w))H_+\left(\frac{z+w}{2}\right) +
    (1-(z+w))H_-\left(\frac{z-w}{2}\right)
  \end{equation}
  is holomorphic in $\mathbb D^2$ and restricts to $H_{\pm}$ on
  $\mathcal V_{\pm}$.
\end{proof}

Since polynomials are dense in the disk algebra $\mathcal A(\mathbb
D)$, an immediate consequence is that polynomials in $z,w$ are dense
in $\mathcal A(\mathcal V)$. It is also evident that $\|F\|_{\mathcal
  V}=\max(\|H_+\|_\infty,\|H_-\|_\infty)$, and that $\mathcal
A(\mathcal V)$ is a uniform algebra with Shilov boundary $\partial
\mathcal V$.

\begin{proposition}
  \label{prop:V-herglotz}
  Let $\mu$ be a finite, nonnegative $n\times n$ matrix valued measure
  on $\partial\mathcal V$.  Then there is a function $F\in
  M_n(Hol(\mathcal V))$ such that $\mu=\text{Re }F$ on $\partial V$ if
  and only if $\mu(\partial \mathcal V_+)=\mu(\partial\mathcal V_-)$.
\end{proposition}

\begin{proof}
  This is more or less immediate from the foregoing description of the
  holomorphic functions on $\mathcal V$; indeed the necessity of the
  condition $\mu(\partial\mathcal V_+)= \mu(\partial\mathcal V_-)$ is
  evident since by restricting to each disk $\mu(\partial \mathcal
  V_\pm)=F_\pm(0)$. Conversely, suppose this constraint holds. Let
  $\mu_{\pm}$ denote the restriction of $\mu$ to the respective
  boundary circles. On each of the disks $\mathcal V_{\pm}$ there is a
  holomorphic function $H_{\pm}$, real-valued at the origin, such that
  $\mathrm{Re}\,H_{\pm} = \mu_{\pm}$ on $\partial \mathcal V_{\pm}$,
  and we have $H_+(0)=H_-(0)=\mu(\partial \mathcal V_{\pm})$.  Thus by
  the lemma $H_{\pm}$ are the restrictions to $\mathcal V_{\pm}$ of a
  function $F$ holomorphic on $\mathcal V$.
\end{proof}

One immediate consequence of this proposition is that a continuous
real-valued function $u$ on $\partial \mathcal V$ is the real part of
the boundary values of a holomorphic function on $\mathcal V$ if and
only if $\int_{\mathcal V_+} u\, dm =\int_{\mathcal V_-}u\, dm$.
Using this fact and the density of polynomials in $\mathcal A(\mathcal
V)$, it follows that, viewing $\mathcal A(\mathcal V)$ as a subalgebra
of $C(\partial \mathcal V)$, the closure of $\mathrm{Re}\,\mathcal
A(\mathcal V)$ in $C_{\mathbb R}(\partial \mathcal V)$ is equal to the
pre-annihilator of the measure $m_+-m_-$ on $\partial \mathcal V$
(here $m_\pm$ is Lebesgue measure on $\partial \mathcal V_\pm$).  Thus
the closure of $\mathrm{Re}\,\mathcal A(\mathcal V)$ in $C_{\mathbb
  R}(\partial \mathcal V)$ has codimension $1$ (in particular
$\mathcal A(\mathcal V)$ is a hypo-Dirichlet algebra on $\partial
\mathcal V$ as claimed). Note that the codimension is also $1$ in the
case of the annulus.

Consider the collection of $n\times n$ matrix-valued holomorphic
functions on $\mathcal V$ with positive real part, normalized to
$F(0)=I_n$, and let $\Gamma_n$ denote the extreme points of this set.
Using Proposition~\ref{prop:V-herglotz}, the results of \cite{BH,BH2}
are again applicable and exactly as in the case of the annulus, every
extreme point of $\Gamma_n$ has the form $F_\mu$ for some finitely
supported $\mu$ of the form
\begin{equation}
  \label{eqn:V-mu-def}
  \mu=\sum_{j=1}^m A_j \delta_{\alpha_j}+\sum_{j=1}^m B_j
  \delta_{\beta_j}
\end{equation}
with $\{\alpha_1, \dots \alpha_m\}, \{\beta_1, \dots \beta_m\}$
subsets of $\partial \mathcal V_\pm$ respectively, and $\sum A_j=\sum
B_j=I$ (though again not every such $F_\mu$ is an extreme point).

We say that $\overline{\mathcal V}$ is a {\em spectral set} for the
pair of commuting operators $S,T$ if the joint spectrum of $S,T$ lies
in $\overline{\mathcal V}$ and, for every polynomial $p(z,w)$, we have
\begin{equation}\label{eqn:V-spectral-def}
  \|p(S,T)\|\leq \|p\|_{\mathcal V}.
\end{equation}
Note that this condition forces $S^2=T^2$, since $p(z,w)=z^2-w^2$
vanishes on $\mathcal V$.

We say that the pair $(S,T)$ acting on the Hilbert space $H$ has a
{\em normal $\partial\mathcal V$ dilation} if there exists a pair of
commuting normal operators $U,V$ acting on a Hilbert space $K$ with
spectrum in $\partial\mathcal V$ and an isometry $\iota:H\to K$ such
that $p(S,T) =\iota^* p(U,V)\iota$ for all polynomials $p$.  
By the definition of $\mathcal V$ and the spectral theorem, the
commuting normal pairs with spectrum in $\partial \mathcal V$ are
precisely the pairs of unitary operators $U,V$ satisfying $U^2=V^2$.

\begin{proposition}
  \label{prop:V-just-check-extremals}
  Let $S,T$ be a pair of commuting operators with joint spectrum in
  $\mathcal V$ and suppose $\overline{\mathcal V}$ is a spectral set
  for $S,T$. Then $(S,T)$ has a normal $\partial \mathcal V$ dilation
  if and only if
  \begin{equation}\label{eqn:V-dilation-condition}
    F_\mu(S,T)+F_\mu(S,T)^*\succeq 0
  \end{equation}
  for all $\mu$ as in (\ref{eqn:V-mu-def}).  
\end{proposition}

\begin{proof}
  If a dilation exists, then (\ref{eqn:V-dilation-condition}) holds by
  the spectral theorem. Conversely, suppose
  (\ref{eqn:V-dilation-condition}) holds.  Then by the Choquet
  integral arguments of \cite{BH,BH2}, and the fact that the joint
  spectral radius of $S,T$ is strictly less than $1$, we have that
  $F(S,T)+F(S,T)^*\succeq 0$ for all matrix-valued functions $F$ on
  $\mathcal V$ with positive real part.  In particular, if $P$ is a
  matrix-valued polynomial with $\|P\|_{\mathcal V}<1$, then
  $F=(I+P)(I-P)^{-1}$ has positive real part, so
  $F(S,T)+F(S,T)^*\succeq 0$ and thus $\|P(S,T)\|\leq 1$. This says
  that the map $p\to p(S,T)$ is completely contractive.
\end{proof}

\begin{proof}[Proof of Theorem~\ref{thm:V-dilation}] 
  Let $\pi$ be a contractive representation of $\mathcal A(\mathcal
  V)$, then $\pi$ is determined by a pair of commuting contractions
  $S,T$ satisfying $S^2=T^2$.  To prove that $\pi$ is completely
  contractive, observe that we may replace $S,T$ by $rS,rT$ for $r<1$.
  (Note that $r^2S^2=r^2T^2$ so $rS,rT$ still determine a homomorphism
  $\pi_r$ of $\mathcal A(\mathcal V)$.)  Indeed, since the map
  $p(z,w)\to p(rz,rw)$ is completely contractive on $\mathcal A(V)$,
  if the maps $\pi_r$ are completely contractive then so is $\pi$. So,
  now that $\|S\|,\|T\|<1$, it suffices to verify the condition of
  Proposition~\ref{prop:V-just-check-extremals}.  But the proof now
  reduces to one essentially identical to the proof given for the
  annulus above; the boundary components $\partial \mathcal V_\pm$
  playing the roles of $\partial_0,\partial_1$.  The only modification
  is to Lemma~\ref{lem:perturb}. In particular when we speak of
  extending $G$ to a neighborhood of a boundary point it should be
  understood that this is a neighborhood in the union of the planes
  $z\pm w=0$ (the full variety $z^2=w^2$ in $\mathbb C^2$); all
  zero-counting is done here. The proof of
  Theorem~\ref{thm:annulus-dilation} then goes through unchanged.
\end{proof}

The question still remains of which pairs $S,T$ with $S^2=T^2$ have
$\mathcal V$ as a spectral set. It is evident that $S$ and $T$ must be
contractions, but this alone is not sufficient. As in the case of the
annulus, there is a one-parameter family of conditions that must be
checked.

\begin{theorem}
  \label{thm:V-spectral-set}
  Let $S,T$ be commuting operators with $S^2=T^2$.  Then $\mathcal V$
  is a spectral set for $S,T$ if and only if
  \begin{equation}
    \|\lambda S+(1-\lambda)T\|\leq 1    
  \end{equation}
  for every complex number $\lambda$ lying on the circle
  $|\lambda-\frac12|=\frac12$.
\end{theorem}

\begin{proof} 
  For the $\lambda$ described in the theorem one may check that the
  functions
  \begin{equation}
    \lambda z+(1-\lambda)w
  \end{equation}
  are bounded by $1$ on $\mathcal V$, so the condition is necessary. 

  Conversely, using again the Choquet integral arguments of
  \cite{BH,BH2} for $\mathcal V$ to be a spectral set it suffices to
  check that $\mathrm{Re}\,F_\mu(S,T)\succeq 0$ for every extreme
  point $F_\mu$ of the set of functions of positive real part on
  $\mathcal V$ (normalized to $F(0,0)=1$). From \cite{BH,BH2} we also
  know that the $\mu$ representing these functions are precisely those
  that put a single unit point mass on each boundary component.  By
  the description of $\mbox{Hol}\mathcal V$ in
  Lemma~\ref{lem:V-extension}, these are the functions whose
  restrictions satisfy
  \begin{equation}
    F_+(t)=\frac{1+\alpha t}{1-\alpha t}, \quad F_-(t)=\frac{1+\beta
      t}{1-\beta t}  
  \end{equation}
  for unimodular constants $\alpha, \beta$. Taking Cayley transforms
  $f_\pm = (F_\pm-1)(F_\pm+1)^{-1}$ we get simply the functions
  \begin{equation}
    f_+(t)=\alpha t, f_-(t)=\beta t
  \end{equation}
  and we require $\|f(S,T)\|\leq 1$ for all $\alpha, \beta$ where $f$
  is any function on $\mathbb D^2$ with $f|_{\mathcal V_\pm}=f_\pm$.
  Multiplying $f$ by $\alpha^*$, we may assume $\alpha=1$, and now it
  is straightforward to check that, putting
  $\lambda=\frac{1+\beta}{2}$, the functions
  \begin{equation}
    f(z,w)=\lambda z+(1-\lambda)w  
  \end{equation}
  do the job. 
\end{proof}

Combining Theorems~\ref{thm:V-dilation} and \ref{thm:V-spectral-set}
we have:

\begin{corollary}
  \label{cor:V-dilation}
  Let $S,T$ be commuting operators on Hilbert space with $S^2=T^2$.
  Then $S,T$ dilate to a commuting pair of unitaries $U,V$ satisfying
  $U^2=V^2$ if and only if
  \begin{equation}
    \|\lambda S+(1-\lambda)T\|\leq 1
  \end{equation}
  for every complex number $\lambda$ on the circle
  $|\lambda-\frac12|=\frac12$.
\end{corollary}

\end{document}